\documentclass[11pt,leqno]{amsart}%
\usepackage{amsmath}
\usepackage{amsfonts}
\usepackage{amssymb}
\usepackage{graphicx}
\usepackage[margin=3.80cm]{geometry}
\providecommand{\U}[1]{\protect\rule{.1in}{.1in}}
\newtheorem{theorem}{Theorem}[section]
\newtheorem*{acknowledgement*}{Acknowledgement}

\newtheorem{corollary}[theorem]{Corollary}

\newtheorem{definition}[theorem]{Definition}
\newtheorem{example}[theorem]{Example}

\newtheorem{lemma}[theorem]{Lemma}

\newtheorem{proposition}[theorem]{Proposition}
\newtheorem{remark}[theorem]{Remark}

\newtheorem*{thmA}{Theorem \ref{LfStabfMin}}
\newtheorem*{thmB}{Theorem \ref{fIndFinfMin}}
\newtheorem*{thmC}{Theorem \ref{mainThTopo}}
\newtheorem*{coroD}{Corollary \ref{mainCoroTopo}}

\begin{document}
\title[$f$--minimal hypersurfaces]{Stability properties and topology at infinity of $f$--minimal hypersurfaces}

\begin{abstract}
We study stability properties of $f$--minimal hypersurfaces isometrically immersed in weighted manifolds with 
non--negative Bakry--\'Emery Ricci curvature under volume growth conditions. Moreover, exploiting a weighted version 
of a finiteness result and the adaptation to this setting of Li--Tam theory, we investigate the topology at 
infinity of $f$--minimal hypersurfaces. On the way, we prove a new comparison result in weighted geometry and we 
provide a general weighted $L^1$--Sobolev inequality for hypersurfaces in Cartan--Hadamard weighted manifolds, 
satisfying suitable restrictions on the weight function. 
\end{abstract}


\date{\today}

\author {Debora Impera}
\address{Dipartimento di Matematica e Applicazioni\\
Universit\`a degli Studi di Milano Bicocca\\
via Cozzi 55\\
I-20125 Milano, ITALY}
\email{debora.impera@gmail.com}
\author {Michele Rimoldi}
\address{Dipartimento di Matematica e Applicazioni\\
Universit\`a degli Studi di Milano Bicocca\\
via Cozzi 55\\
I-20125 Milano, ITALY}
\email{michele.rimoldi@gmail.com}
\maketitle
\tableofcontents

\section{Introduction}

Many problems in geometric analysis lead to consider Riemannian manifolds endowed with a measure that has a smooth 
positive density with respect to the Riemannian one. This turns out to be compatible with the metric structure of 
the manifold and the resulting spaces take the name of weighted manifolds, also known in the literature as manifolds 
with density. Weighted manifolds first arose in the study of diffusion processes on manifolds in works of D. Bakry 
and M. \'Emery, \cite{BE}, and were intensively studied in recent years; see e.g. the seminal works of F. Morgan, 
\cite{Mo-notices}, and G. Wei, W. Wylie, \cite{WW}. A weighted manifold is a triple 
$M^m_f=(M^m, \left\langle \,,\,\right\rangle, e^{-f}d\mathrm{vol}_M)$, where $(M^m, \left\langle \,,\,\right\rangle)$ 
is a complete $m$--dimensional Riemannian manifold, $f\in C^{\infty}(M)$ and $d\mathrm{vol}_M$ denotes the canonical 
Riemannian volume form on $M$. The geometry of weighted manifolds is visible in the weighted metric structure, i.e., in the weighted 
measure of (intrinsinc) metric objects, and it is controlled by suitable concepts of curvature adapted to the density of the measure. In \cite{BE} (see also \cite{Lic}), it was introduced an important generalization of Ricci curvature in this setting, known as  Bakry--\'Emery Ricci tensor and defined as
\[
\ \mathrm{Ric}_f=\mathrm{Ric}+\mathrm{Hess}(f).
\]
Following M. Gromov, \cite{Gro}, if we consider an isometrically immersed orientable hypersurface 
$\Sigma^m$ in the weighted manifold $M_f$, we can also define a generalization of the mean curvature vector 
field as
\[
\ \mathbf{H}_f=\mathbf{H}+(\overline{\nabla}f)^{\bot}.
\]
Here we have denoted by $\mathbf{H}$ the mean curvature vector field of the immersion, by $\overline{\nabla}$ the 
Levi--Civita connection on $M$, and by $(\cdot)^{\bot}$ the projection on the normal bundle of $\Sigma$.

It is a well--known fact that minimal hypersurfaces arise as critical points of the area functional. Since the weighted structure on $M$ induces a weighted structure on $\Sigma$ we can consider the variational problem for the weighted area functional
\[
\ \mathrm{vol}_{f}(\Sigma)=\int_{\Sigma}e^{-f}d\mathrm{vol}_{\Sigma}. 
\]
From variational formulae, \cite{Bay}, one can see that $\Sigma$ is $f$--minimal, namely a critical point of the 
weighted area functional, if and only if $\mathbf{H}_f$ vanishes identically.

Clearly, minimal hypersurfaces are a particular case of $f$--minimal hypersurfaces, corresponding to the case $f\equiv const.$
Moreover, as we shall see more in details in Section \ref{DefEx}, self--shrinkers of the mean curvature flow are important examples of $f$--minimal hypersurfaces in the Euclidean space with the Gaussian density $e^{-|x|^2/2}$. This, on one hand, gives a motivation to the study of $f$--minimal hypersurfaces and, on the other hand, strongly suggests to study self--shrinkers in the realm of weighted manifolds; this is the point of view adopted in \cite{Rim}, \cite{PR2}.

The research on $f$--minimal hypersurfaces has just started and it has been already approached by many authors; 
see e.g. \cite{Fan}, \cite{Ho}, \cite{Si}, \cite{ChMeZh1}, \cite{ChMeZh2}, \cite{L}, \cite{Esp}, \cite{SYu}. Much effort has 
been devoted to the study of the stability properties. As we will see later on, the stability properties of 
$f$--minimal hypersurfaces are taken into account by spectral properties of the following weighted Jacobi operator
\[
\ L_f=-\Delta_f-\left(|\mathbf{A}|^2+\overline{\mathrm{Ric}}_{f}(\nu,\nu)\right),
\]
where $\mathbf{A}$ denotes the second fundamental form of the immersion, $\overline{\mathrm{Ric}}_{f}$ denotes the Bakry--\'Emery Ricci tensor of the ambient space, and $\Delta_f=\Delta-\left\langle \nabla f, \nabla\,\cdot\right\rangle$ is the $f$--laplacian on $\Sigma_f$. Roughly speaking (for more details see Section \ref{SectStab} below) we say that an $f$--minimal hypersurface is $L_f$--stable if it minimizes the weighted area functional. The most up to date result, proved by X. Cheng, T. Mejia, and D. Zhou, \cite{ChMeZh1}, states that there exist no $L_f$--stable complete $f$-- minimal hypersurfaces $\Sigma$ immersed in a complete weighted manifold $M_f$ with $\overline{\mathrm{Ric}}_f\geq k>0$, provided $\mathrm{vol}_f(\Sigma)<+\infty$. Note that, by the equivalences obtained in \cite{ChZh}, in the special case of self--shrinkers this conclusion was originally pointed out by T. Colding and W. Minicozzi in \cite{CoMi}.

In the first part of the paper, we are able to generalize the result in \cite{ChMeZh1}, considering 
progressively weaker growth conditions on the intrinsic weighted volume growth of geodesic balls. Recall that, if $B_{r}\left(  o\right)  $ and $\partial
B_{r}\left(  o\right)  $ denote respectively the metric ball and the metric sphere of $\Sigma$ of radius $r>0$ and centered
at $o\in \Sigma$, we define%
\[
\mathrm{vol}_{f}\left(  B_{r}\left(  o\right)  \right)  =\int_{B_{r}\left(
o\right)  }e^{-f}d\mathrm{vol}_\Sigma,\qquad\mathrm{vol}_{f}\left(  \partial
B_{r}\left(  o\right)  \right)  =\int_{\partial B_{r}\left(  o\right)  }%
e^{-f}d\mathrm{vol}_{m-1},
\]
where $d\mathrm{vol}_{m-1}$ stands for the $\left(  m-1\right)$--Hausdorff measure. We then prove the following theorem.
\begin{thmA}
Let $M_f$ be a complete weighted manifold with $\overline{\mathrm{Ric}}_f\geq k>0$. Then there is no 
$L_f$--stable complete non--compact $f$--minimal hypersurface $\Sigma$ immersed in $M_f$ provided $\mathrm{vol}_f(B_r(o))=O(\mathrm{e}^{\alpha r})$ as $r\rightarrow+\infty$, with $\alpha<2\sqrt{k}$.
\end{thmA}

Furthermore, in the instability case, exploiting the oscillatory behaviour of solutions of some 
ODEs that naturally arise in this setting, we investigate general geometric restrictions for the 
finiteness of the weighted index of the $f$--minimal hypersurface, that is, the maximum dimension of 
the linear space of compactly supported deformations that decrease the weighted area up to second 
order. 

\begin{thmB}
Let $M_f$ be a complete weighted manifold with $\overline{\mathrm{Ric}}_f\geq k>0$. Then there is no complete $f$--minimal hypersurface $\Sigma$ immersed in $M_f$ with  $Ind_f(\Sigma)<+\infty$ provided one of the following conditions hold
 \begin{enumerate}
  \item $\mathrm{vol}_f(\Sigma)=+\infty$ and $\mathrm{vol}_f(B_r(o))\leq C r^a$ for any $r\geq r_0$ and some positive constants $C$, $r_0$ and $a$;
  \item  $\mathrm{vol}_f(\partial B_r)^{-1}\notin L^{1}(+\infty)$ and $|A|\notin L^{2}(\Sigma, e^{-f}d\mathrm{vol}_\Sigma)$.
\end{enumerate}

\end{thmB}

Note that this last research direction is significant also in the special case of self--shrinkers. 
We are not aware of any result in this direction up to now.

The second aim of this paper is to obtain information on the topology at infinity of $f$--minimal hypersurfaces immersed in suitable ambient spaces. We recall that, in the non--weighted setting, there is a well--known connection, developed by P. Li and L.--F. Tam and collaborators (see e.g. \cite{LT}), between the dimension of the space of $L^2$--harmonic forms, the number of non--parabolic ends, and the Morse index of the operator $-\Delta-a(x)$, where $-a(x)$ is the smallest eigenvalue of the Ricci tensor at $x$. Furthermore, following H. D. Cao, Y. Shen, S. Zhu, \cite{CSZ}, and P. Li and J. Wang, \cite{LW1}, one shows that if the manifold supports a $L^1$--Sobolev inequality outside some compact set, then all ends are non--parabolic. According to D. Hoffman and J. Spruck, \cite{HoSp}, this in particular applies to minimal submanifolds of Cartan--Hadamard manifolds.

In this order of ideas, by adapting the Li--Tam theory to the weighted setting and by means of a weighted version of an abstract finiteness result from \cite{PRS_Progress}, we are able to obtain the finiteness of the number of non--$f$--parabolic ends of a weighted manifold $M_f$, assuming the finiteness of the Morse index of the operator $-\Delta_f-a(x)$, where $-a(x)$ is now the smallest eigenvalue of $\mathrm{Ric}_f$ at $x$.

Using then the technique adopted in \cite{MiSi}, \cite{HoSp}, we are able to guarantee the validity of a weighted $L^1$--Sobolev inequality outside some compact set on $f$--minimal hypersurfaces with finite weighted index, under suitable assumptions on $f$ and on the curvature of the ambient weighted manifold. On the way we prove a comparison theorem in weighted geometry assuming an upper bound on the sectional curvature. An adaptation to the weighted setting of the results in \cite{CSZ}, \cite{LW1} finally provides the following topological result.

\begin{thmC}
Let $\Sigma^m$ be a complete $f$--minimal hypersurface isometrically immersed with $Ind_f(\Sigma)<+\infty$ in a complete weighted manifold $M_{f}^{m+1}$ with $\overline{Sect}\leq0$ and $\overline{\mathrm{Ric}}_{f}\geq k\geq0$. Suppose furthermore that $f^{*}=\sup_\Sigma f<+\infty$ and $|\overline{\nabla} f|\in L^{m}(\Sigma_f)$. Then $\Sigma$ has finitely many ends.
\end{thmC} 

As a consequence, adapting ideas in \cite{LiYauCMP},  we are able to obtain the following result, in which we replace the finiteness of the weighted index with the finiteness of the weighted total curvature of the $f$--minimal hypersurface.

\begin{coroD}
Let $\Sigma^m$ be a complete $f$--minimal hypersurface isometrically immersed in a complete weighted manifold $M_{f}^{m+1}$ with $\overline{Sect}\leq0$ and $\overline{\mathrm{Ric}}_{f}\geq k\geq0$. Assume that $|\mathbf{A}|\in L^{m}(\Sigma_f)$. Suppose furthermore that $f\leq f^{*}<+\infty$ and $|\overline{\nabla} f|\in L^{m}(\Sigma_f)$. Then $\Sigma$ has finitely many ends.
\end{coroD}

The paper is organized as follows. In Section \ref{DefEx} we introduce some notation and provide some examples of $f$--minimal hypersurfaces. Section \ref{SectStab} is devoted to the study of stability properties of $f$--minimal hypersurfaces. Namely we analyze geometric conditions for the instability and infiniteness of the weighted index of these objects. In Section \ref{Fin&Li-Tam} we present a weighted version of an abstract finiteness result, recently obtained in \cite{PRS-RevMatIb}, and state the adapted Li--Tam theory in the weighted setting. In Section \ref{Compar} we prove a new comparison result in weighted geometry. In Section \ref{SobWeight} a proof of the weighted $L^1$--Sobolev inequality for hypersurfaces in Cartan--Hadamard manifolds is provided. We end the paper with Section \ref{TopoRes}, where we finally prove the topological Theorem \ref{mainThTopo} and Corollary \ref{mainCoroTopo}.

\section{Definitions and some examples}\label{DefEx}
Let $M_f^{m+1}=(M^{m+1},\langle\cdot\,,\,\cdot \rangle, \mathrm{e}^{-f}d\mathrm{vol}_M)$ be a weighted manifold and 
let $\Sigma^m$ be an isometrically immersed orientable hypersurface. We will denote by $\mathbf{A}$ the second 
fundamental form of the immersion $x:\Sigma^m\rightarrow M_f^{m+1}$, that is
\[
\mathbf{A}(X,Y)=(\overline{\nabla}_{X}Y)^{\bot}, 
\]
where $\overline{\nabla}$ denotes the Levi-Civita connection on $M$ and $(\cdot)^{\bot}$ denotes the projection on the normal bundle of $\Sigma$.
Denote by $\mathbf{H}=\mathrm{tr}_{\Sigma}\mathbf{A}$ the mean curvature vector field of the immersion. We define the $f$--\textit{mean 
curvature vector field} of $\Sigma$ as
\[
\mathbf{H}_f=\mathbf{H}+(\overline{\nabla}f)^{\bot}. 
\]
Hence, denoting by $\nu$ be the unit normal we define the \emph{$f$--mean curvature} $H_f$ of $\Sigma$ by $\mathbf{H}_f:=H_f\nu$
\begin{definition}
Let $x:\Sigma^m\rightarrow M_f^{m+1}$ be a connected isometrically immersed hypersurface. We say that $\Sigma$ is $f$--\textit{minimal} if $\mathbf{H}_f\equiv0$.
\end{definition}
\begin{remark}
\rm{Note that, when $f$ is constant $\mathbf{H}_f=\mathbf{H}$ and we recover the usual definition of a minimal hypersurface.}
\end{remark}

\begin{example}(Self--shrinkers)
\rm{Let $\Sigma^m$ be a complete $m$--dimensional Riemannian manifold without boundary smoothly 
immersed by $x_0:\Sigma^m\to\mathbb{R}^{m+1}$ as an hypersurface in the Euclidean 
space $\mathbb{R}^{m+1}$. We say that $\Sigma_0=x_0(\Sigma^m)$ is moved along its  mean curvature 
vector if there  is a $1$--parameter family of smooth immersions $x:\Sigma^m\times[t_{0},T)\to\mathbb{R}^{m+1}$, with  corresponding 
hypersurfaces $\Sigma_t=x(\cdot\,,\,t)(\Sigma^m)$, such that it satisfies the following mean curvature flow initial value  problem
\begin{equation}\label{MCF}
\begin{cases}
\frac{\partial}{\partial t}x(p,t)=\mathbf{H}(p,t)&\\
x(\cdot,\,t_0)=x_0,&
\end{cases} 
\end{equation}
for any $p\in \Sigma^m$, $t\in[t_{0}, T)$.
Here $\mathbf{H}(p,t)$ is the mean curvature vector field 
of the hypersurface $M_t$ at $x(p,t)$. The short time existence and uniqueness of a solution of 
\eqref{MCF} was investigated in classical works on quasilinear parabolic equations.
  
A MCF $\{\Sigma_t\}_{t<0}$ is called a  self--shrinking solution if it satifies
\[
\ \Sigma_t=\sqrt{-2t}\Sigma_{-\frac12}
\]
for all $t<0$.
For an  overview on  the role that such solutions play in the study of MCF  see e.g. the 
introduction   in \cite{CoMi}. An hypersurface is said to be a self--shrinker if it is the time $t=-\frac{1}{2}$ slice of a self--shrinking solution. Equivalently, by a self shrinker  (based at $0\in\mathbb{R}^{m+1}$) we  mean  a connected, isometrically  immersed hypersurface
$x:\Sigma^{m}\rightarrow\mathbb{R}^{m+1}$ whose mean curvature vector field  satisfies the equation 
\begin{equation}\label{SS}
x^{\bot}=-\mathbf{H} .
\end{equation}

Let $f=\frac{|x|^2}{2}$ and consider the Gaussian space $\mathbb{R}^{m+1}_f$, 
which is the Euclidean space endowed with the 
canonical metric and the measure $\mathrm{e}^{-|x|^2/2}d\mathrm{vol}_{\mathbb{R}^{m+1}}$.
A simple computation shows that
\[
\overline{\nabla}f=x, 
\]
We hence obtain that $f$--minimal hypersurfaces $\Sigma$ in the Gaussian space $\mathbb{R}_f^{m+1}$ satisfy
\[
\mathbf{H}+ x^{\bot}=0, 
\]
and thus are exactly the self--shrinkers of mean curvature flow.
}
\end{example}

\begin{example}\label{es2}(Slices of warped products of the form $P\times_{e^{-f}}\mathbb{R}$)
\rm{Let $M^{m+1}=P^m\times_{e^{-f}}\mathbb{R}$, where $P$ is an $m$--dimensional Riemannian manifold, 
$f:P\rightarrow \mathbb{R}_+$ is a smooth function and the product manifold $P\times\mathbb{R}$ is endowed with the 
Riemannian metric
$$
\langle\cdot\,,\,\cdot\rangle=\pi_{P}^*(\langle\cdot\,,\,\cdot\rangle_{P})+e^{-2f(\pi_{P})}\pi_{\mathbb{R}}^{*}(d t\otimes dt).
$$
Here $\pi_{\mathbb{R}}$ and $\pi_{P}$ denote the projections onto the corresponding factors and $\langle\cdot\,,\,\cdot\rangle_{P}$ is the 
Riemannian  metric on $P^m$.
It is  a well--known fact (see for instance \cite{O}) that the distribution on the 
space orthogonal to $T=\partial/\partial t$ provides a foliation of $M$ by means of 
totally geodesic (hence minimal) leaves $P_t=P\times \left\{t\right\} $, $t\in \mathbb{R}$. 
Moreover, since the function $f$ only depends on $P$, it follows that the unit normal $\nu_t=T$, is everywhere 
orthogonal to $\overline{\nabla} f$. Hence the slices $P_t$, $t\in \mathbb{R}$, represent a distinguished family
of $f$--minimal hypersurfaces in $M$. 
}
\end{example}

\section{Stability properties}\label{SectStab}
It is a well--known fact that minimal hypersurfaces arise from a variational problem. Indeed, they 
are critical points of the area functional
\[
\mathrm{vol}(\Sigma)=\int_\Sigma d\mathrm{vol}_{\Sigma}. 
\]
More precisely, letting $x_t$ , $t\in(-\varepsilon,\varepsilon)$,
$x_0 = x$, be a smooth compactly supported variation of immersions and denoting by
$V$ the associated variational vector field along $x$ one gets that
\[
\frac{d}{dt} \mathrm{vol}(x_t(\Sigma))\Big|_{t=0}=-\int_\Sigma \langle \mathbf{H}, V\rangle d\mathrm{vol}_{\Sigma}. 
\]
A similar characterization can be given also for $f$--minimal hypersurfaces, 
(see e.g \cite{Si}, \cite{ChMeZh1}). Indeed, defining the weighted area functional of $\Sigma^m\to M^{m+1}_f$ by
\[
\ \mathrm{vol}_f(\Sigma)=\int_{\Sigma}e^{-f}d\mathrm{vol}_{\Sigma}
\]
we have that
\[
\frac{d}{dt} \mathrm{vol}_f(x_t(\Sigma))\Big|_{t=0}=
-\int_\Sigma \langle \mathbf{H}_f, V\rangle e^{-f}d\mathrm{vol}_{\Sigma}. 
\]

We can now give the following
\begin{definition}
Let $x_t$ , $t\in(-\varepsilon,\varepsilon)$,
$x_0 = x$, be a smooth compactly supported variation of immersions. We say that a $f$--minimal hypersurface $\Sigma$ is 
$L_f$--stable if
\[
\frac{d^2}{dt^2} \mathrm{vol}_f(x_t(\Sigma))\Big|_{t=0}\geq 0.
\]
\end{definition}
Denote by $V$ the variational vector field along $x$ associated to the variation and let $V=u\nu$, $u\in C_{c}^{\infty}$. By a direct computation one can prove the following second variation formula for the weighted area, \cite{Bay},
\begin{align*}
\frac{d^2}{dt^2} \mathrm{vol}_f(x_t(\Sigma))\Big|_{t=0}=&
\int_\Sigma (|\nabla u|^2-u^2(|\mathbf{A}|^2+
\overline{\mathrm{Ric}}_f(\nu,\nu)))\mathrm{e}^{-f}\,d\mathrm{vol}_{\Sigma}\\
=& \int_\Sigma u\,L_f u\,\mathrm{e}^{-f}\,d\mathrm{vol}_{\Sigma},
\end{align*}
where $\overline{\mathrm{Ric}}_{f}$ is the Bakry--\'Emery Ricci tensor of $M_{f}^{m+1}$, and the operator $L_f$ is defined by
\[
L_f u=-\Delta_f u-(|\mathbf{A}|^2+\overline{\mathrm{Ric}}_f(\nu,\nu))u.
\]
Some steps into the study of non--existence results for $L_f$ stable $f$--minimal hypersurfaces were moved in \cite{Fan}, \cite{ChMeZh1}, \cite{L}.

\begin{proposition}[\cite{Fan}, Corollary 1.4, \cite{L}, Theorem 1]\label{NoLfStabCpt}
Let $M_f$ be a complete weighted manifold with $\overline{\mathrm{Ric}}_f\geq k$
and let $\Sigma$ be a compact $f$--minimal hypersurface immersed in $M_f$ . 
\begin{itemize}
\item[(a)] If $k>0$ then $\Sigma$ cannot be $L_f$--stable;
\item[(b)] If $k=0$ and $\Sigma$ is $L_f$--stable, then it has to be 
totally geodesic and $\overline{\mathrm{Ric}}_f(\nu,\nu)=0$.  
\end{itemize}
\end{proposition}

\begin{proposition}[\cite{ChMeZh1}, Theorem 5]\label{NoLfStabfVolFin}
Let $M_f$ be a complete weighted manifold with $\overline{\mathrm{Ric}}_f\geq k>0$. Then there exists no 
$L_f$--stable complete $f$--minimal hypersurface $\Sigma$ immersed in $M_f$ without boundary and with $\mathrm{vol}_f(\Sigma)<+\infty$.
\end{proposition}

\begin{remark}
\rm{When $\Sigma$ is a complete self--shrinker with $\mathrm{vol}_f(\Sigma)<+\infty$, the conclusion in Proposition 
\ref{NoLfStabfVolFin} was originally pointed out by T. Colding an W. Minicozzi in \cite{CoMi}. This follows by the 
equivalences obtained in \cite{ChZh}. Note also that it was conjectured by H. D. Cao that the weighted volume of 
complete self--shrinkers is always finite. On the other hand there is still no real evidence for this conjecture; 
see the very recent \cite{PR2} where some steps in this direction are made.}
\end{remark}

Following classical terminology in linear potential theory recall that
a weighted manifold $M_f=\left(M,\left\langle \cdot\,,\,\cdot\right\rangle
,e^{-f}d\mathrm{vol}_{M}\right)$ is said to be $f$--parabolic if every solution of
$\Delta_{f}u\geq0$ satisfying $u^{\ast}=\sup_{M}u<+\infty$ must be identically constant.

Moreover, the positivity of a Schr\"{o}dinger operator can be formulated in term of the existence of positive 
solutions of the associated linear equation. Indeed the following equivalence holds, that is  
a weighted version of 
a classical result by D. Fischer--Colbrie and R. Schoen, \cite{FiCoSh}, and W. F. Moss 
and J. Piepenbrink, \cite{MP}.

\begin{proposition}[\cite{Sa}]\label{weightedFCS}
Let $M_f$ be a weighted manifold, and $\Omega\subset M$ be a domain in $M$ and let $L=-\Delta_f+q\left(x\right)$, $q\left(x\right)\in L^\infty_{loc}\left(\Omega\right)$. Denote by $\lambda_{1}^{L}(\Omega)$ the bottom of the spectrum of $L$ on $\Omega$. The following facts are equivalent.
\begin{enumerate}
	\item [(i)]There exist $\omega\in C^{1,\alpha}_{loc}\left(\Omega\right)$, $\omega>0$, weak solution of
	\[
	\ \Delta_f\omega-q\left(x\right)\omega=0\quad\textrm{on}\quad\Omega.
	\] 
	\item [(ii)]There exist $\varphi\in W^{1,2}_{loc}\left(\Omega\right)$, $\varphi>0$, weak solution of
	\[
	\ \Delta_f\varphi -q\left(x\right)\varphi\leq0\quad\textrm{on}\quad\Omega.
	\]
	\item [(iii)]$\lambda_1^{L}\left(\Omega\right)\geq 0$.
\end{enumerate}
\end{proposition}
The proof of the previous proposition is straightforward once one looks at Lemma 3.10 in \cite{PRS_Progress} and 
observes that $L=-\Delta_f+q\left(x\right)$ is unitarly equivalent to the Schr\"odinger operator 
\[
\ S=-\Delta+\left[\left(1/4\left\langle \nabla f,\nabla f\right\rangle-1/2\Delta f\right)+q\left(x\right)\right]=-\Delta+\left(p\left(x\right)+q\left(x\right)\right)
\]
under the multiplication map $T=M_{e^{-f/2}}:L^2\left(M, e^{-f}d\mathrm{vol}_{M}\right)\rightarrow L^2\left(M, d\mathrm{vol}_{M}\right)$.

\begin{example}($L_f$--stable $f$--minimal hypersurfaces in warped products)
\rm{
Let $M^{m+1}=P^m\times_{e^{-f}}\mathbb{R}$, as in Example \ref{es2}, and let $\Sigma^m$ be an $f$--minimal
hypersurface isometrically immersed in $M^{m+1}$.
Setting $Y=e^{-f}T$ it is not hard to prove that the function $u=\langle Y, \nu\rangle$ 
satisfies
\[
\Delta_f u+(|\mathbf{A}|^2+\overline{\mathrm{Ric}}_f(\nu,\nu))u=0.
\]
Hence, using the previous proposition, we can see that every $f$--minimal hypersurface $\Sigma$ isometrically immersed in $M$ satisfying $0<u$ is $L_f$--stable.
}
\end{example}

We can now obtain the following generalization of Proposition \ref{NoLfStabCpt} and Proposition \ref{NoLfStabfVolFin}. Note that point (b) was already obtained in the very recent Theorem 7.3 in \cite{Esp}.

\begin{proposition}\label{unstabfpar}
Let $M_f$ be a complete weighted manifold with $\overline{\mathrm{Ric}}_f\geq k\geq0$
and let $\Sigma$ be a $f$--parabolic, complete, $f$--minimal hypersurface immersed in $M_f$.  
\begin{itemize}
\item[(a)] If $k>0$ then $\Sigma$ cannot be $L_f$--stable.
\item[(b)] if $k=0$ and $\Sigma$ is $L_f$--stable, then it has to be totally geodesic and
$\overline{\mathrm{Ric}}_f(\nu,\nu)=0$.  
\end{itemize}
\end{proposition}

\begin{proof}
Assume that $\Sigma$ is a $L_f$--stable complete $f$--minimal hypersurface immersed in $M_f$ 
which is $f$--parabolic. Since $\Sigma$ is $L_f$--stable, it follows by Proposition 
\ref{weightedFCS} that there exists 
a nonconstant function $u\in W^{1,2}_{loc}\left(\Sigma_f\right)$, $u>0$, weak solution of
\[
\Delta_f u+(|\mathbf{A}|^2+\overline{\mathrm{Ric}}_f(\nu,\nu))u=0.
\]
Since $\overline{\mathrm{Ric}}_f$ is bounded below by a positive constant $k$, this also implies that $u$ is a weak solution of
\[
\Delta_f u\leq -(k+|\mathbf{A}|^2)u\leq0.
\]
Hence $u$ is a $f$--superharmonic function bounded from below and, since $\Sigma$ is $f$--parabolic,
it must be constant. In particular,
\[
|\mathbf{A}|^2+\overline{\mathrm{Ric}}_f(\nu,\nu)=0. 
\]
and the conclusion follows immediately.
\end{proof}

\begin{remark}
\label{rem_fvolumegrowth}
{\rm 
It can be shown that a
sufficient condition for $\Sigma$ to be $f$--parabolic is that it is geodesically complete
and%
\begin{equation}\label{fpargeospheres}
\mathrm{vol}_{f}\left(  \partial B_{r}\right)  ^{-1}\notin L^{1}\left(
+\infty\right)  . 
\end{equation}
This fact can be easily established adapting to the diffusion operator
$\Delta_{f}$ standard proofs for the Laplace--Beltrami operator; see \cite{Gr},
\cite{RS}.\\
Moreover, note that it is not difficult to prove that $f$--parabolicity is also guaranteed if we assume the stronger 
condition
\begin{equation}\label{fpargeoballs}
\mathrm{vol}_{f}\left(B_{r}\right)=O(r^2),\quad \mathrm{as\ } r\rightarrow+\infty.
\end{equation}
The previous formula shows also that $f$--parabolicity holds if the manifold $\Sigma$ has finite $f$--volume. Hence, in particular, the conclusion in Proposition \ref{unstabfpar} can be obtained if we either assume $\mathrm{vol}_f(\Sigma)<+\infty$ or $\mathrm{vol}_f(B_{r}(o))=O(r^{2})$ as $r\to+\infty$.
}
\end{remark} 

In the following result we show that one can do even better, assuming a weaker growth condition on the weighted volume of geodesic balls.

\begin{theorem}\label{LfStabfMin}
Let $M_f$ be a complete weighted manifold with $\overline{\mathrm{Ric}}_f\geq k>0$. Then there is no 
$L_f$--stable complete non--compact $f$--minimal hypersurface $\Sigma$ immersed in $M_f$ provided 
$\mathrm{vol}_f(B_r(o))=O(\mathrm{e}^{\alpha r})$ as $r\rightarrow+\infty$, with $\alpha<2\sqrt{k}$.
\end{theorem}

\begin{proof}
Define the weighted volume entropy of 
$(\Sigma, \left\langle \cdot\,,\,\cdot\right\rangle_\Sigma, \mathrm{e}^{-f} d\mathrm{vol}_\Sigma)$ as 
\[
h_f(\Sigma) := \limsup_{r\rightarrow+\infty} \frac{\log \mathrm{vol}_f (B_r(o))}{r}.
\]
As observed in \cite{Br}, the following inequality holds true in general for the bottom of the spectrum of the $f$--Laplacian $\lambda_1^{f}$:
\[
\lambda_1^f(\Sigma)\leq \frac14 h_f^2(\Sigma).
\]
Hence, in particular, if we assume that $\mathrm{vol}_f(B_r(o))=O(\mathrm{e}^{\alpha r})$ as
$r\rightarrow+\infty$ we obtain
\[
\lambda_1^f(\Sigma)\leq \frac{\alpha^2}{4}. 
\]
Now assume by contradiction that $\Sigma$ is $L_f$--stable. Then
\begin{align*}
\frac{\alpha^2}{4}\geq\lambda_1^f(\Sigma)= & 
\inf_{0\neq u\in C^{\infty}_{c}(\Sigma)}\frac{\int_\Sigma|\nabla u|^2\mathrm{e}^{-f}dvol_\Sigma}{\int_\Sigma u^2\mathrm{e}^{-f}d\mathrm{vol}_\Sigma}\\  
\geq & \frac{\int_\Sigma(|A|^2+\overline{\mathrm{Ric}}_f(\nu,\nu))u^2\mathrm{e}^{-f}d\mathrm{vol}_\Sigma}
{\int_\Sigma u^2\mathrm{e}^{-f}d\mathrm{vol}_\Sigma}\\  
\geq & k,
\end{align*}
for any $u\in C^\infty_c(\Sigma)$, contradicting the assumption on $\alpha$.
\end{proof}

In order to study $L_f$--unstable $f$--minimal hypersurfaces we introduce the $f$--index of $\Sigma$ as 
the generalized Morse index of $L_f$ on $\Sigma$. Namely, let $x:\Sigma^m\rightarrow M_f^{m+1}$ be an 
isometrically immersed complete orientable $f$--minimal hypersurface. 
Given a bounded domain $\Omega\subset\Sigma$ we define 
\[
Ind^{L_f} (\Omega)=\# \{\mathrm{negative\ eigenvalues\ of\ } L_f \mathrm{\ on\ } \mathcal{C}_{0}^{\infty}(\Omega) \}.
\]
The $f$--\textit{index} of $\Sigma$ is then defined as
\[
Ind_f(\Sigma):=Ind^{L_f}(\Sigma)=\sup_{\Omega\subset\subset\Sigma}Ind^{L_f} (\Omega). 
\]
Geometrically, the $f$--index of $\Sigma$ can be described as the maximum dimension of the linear space of 
compactly supported deformations that decrease the weighted area up to second order.

The following result, due to B. Devyver, \cite{De}, permits to interpret the finiteness of the Morse index of a Schr\"{o}dinger operator in terms of the existence of a positive solution of the associated linear equation outside a compact set (also in the weighted setting).
\begin{proposition}\label{Baptiste}
Let $\Sigma_f$ be a complete weighted manifold, and let $L=-\Delta_f-q(x)$, $q(x)\in L^{\infty}_{loc}(\Sigma)$. The following facts are equivalent
\begin{enumerate}
	\item[(i)] $L$ has finite Morse index.
	\item[(ii)] There exists a positive smooth function $\varphi\in W_{loc}^{1,2}$ which satisfies $L\,\varphi=0$ outside a compact set.
	\item[(iii)]$\lambda_1^{L}(\Sigma\setminus\Omega)\geq0$, for some $\Omega\subset\subset\Sigma$.
\end{enumerate}
 
\end{proposition}
Let $v(t)=\mathrm{vol}_f(\partial B_t(o))$, where 
$\partial B_t(o)$ are the geodesic spheres of radius $t$ in $\Sigma$. Note that by the co-area formula we have that
\begin{equation}\label{co-area}
\mathrm{vol}_f(B_r(o))=\int_{0}^rv(t)dt.
\end{equation}
We obtain the following
\begin{proposition}\label{fIndDCZ}
Let $\Sigma_f$ be a complete noncompact weighted manifold with $\mathrm{vol}_f(\Sigma)=+\infty$ and let $\Omega$ be an arbitrary compact
subset of $\Sigma$. Then
\begin{enumerate}
 \item If $\mathrm{vol}_f(B_r(o))\leq C r^a$ for any $r\geq r_0$ and some positive constants $C$, $r_0$ and $a$, then
 $\lambda_1^f(\Sigma\backslash \Omega)=0$.
 \item If $\mathrm{vol}_f(B_r(o))\leq C \mathrm{e}^{\alpha r}$ for any $r\geq 0$ and some positive constants $C$ and $\alpha$, then
 $\lambda_1^f(\Sigma\backslash \Omega)\leq \frac{\alpha^2}{4}$.
\end{enumerate}

 \end{proposition}

\begin{proof}
Since $\Omega$ is compact we can find a constant $T_0$ such that $\Omega\subset B_{T_0}(o)$. We reason now as in \cite{DCZhou}, \cite{BMR}, and exploit the oscillatory behaviour under our assumptions of solutions of the ODE
\begin{equation}\label{ODEOsc}
\begin{cases}
(v(t)x^{\prime}(t))^{\prime}+\lambda v(t)x(t)=0,\quad \textrm{a.e. on}\quad (T_0,+\infty),\\
x(T_0)=x_0,
\end{cases}
\end{equation} 
where $v(t)$ is a positive continuous function on $[T_0, +\infty)$ and $\lambda$ is a positive constant. Choosing $v(t)=\mathrm{vol}_f(\partial B_t(o))$ it then follows from Theorem 2.1 in \cite{DCZhou} that equation \eqref{ODEOsc} is oscillatory provided $\Sigma$ has infinite $f$--volume and either the assumption in (1), or $\lambda>\frac{\alpha^2}{4}$ and the assumption in (2), hold true. Now the proof proceeds with slight modifications as in Theorem 3.1 in \cite{DCZhou}, but we report it here for the sake of completeness. Let us first assume that $\mathrm{vol}_f(B_r(o))\leq Cr^a$ for any $r\geq r_0$ and some positive constants $C$, $r_{0}$, $a$. Then for any $\lambda>0$ there exists some nontrivial oscillatory solution $x_\lambda(t)$ of \eqref{ODEOsc} a.e. on $[T_0,+\infty)$, i.e., there exist $T_1^{\lambda}$ and $T_2^{\lambda}$ in $[T_0,+\infty)$ such that $T_1^{\lambda}<T_2^{\lambda}$, $x_{\lambda}(T_1^{\lambda})=x_{\lambda}(T_2^{\lambda})=0$, and $x_{\lambda}(t)\neq 0$ for any $t\in(T_1^{\lambda}, T_2^{
\lambda})$. 
Let $\varphi_\lambda(x)=x_\lambda(r(x))$ and $\Omega_\lambda=B_{T_2^\lambda}(o)\setminus B_{T_1^{\lambda}}(o)$. It follows that
\begin{eqnarray*}
0\leq\lambda_1^f(\Sigma\setminus\Omega)&\leq&\lambda_1^{f}(\Omega_\lambda)\\
&\leq&\frac{\int_{\Omega_\lambda}|\nabla \varphi_\lambda|^2e^{-f}d\mathrm{vol}_{\Sigma}}{\int_{\Omega_\lambda}|\varphi_\lambda|^2e^{-f}d\mathrm{vol}_{\Sigma}}\\
&=&\frac{\int_{T_1^{\lambda}}^{T_2^{\lambda}}(x^{\prime}_{\lambda}(r))^2v(r)dr}{\int_{T_1^{\lambda}}^{T_2^{\lambda}}(x_{\lambda}(r))^2v(r)dr}\\
&=&-\frac{\int_{T_1^{\lambda}}^{T_2^{\lambda}}(v(r)x^{\prime}_{\lambda}(r))^{\prime}x_\lambda(r)dr}{\int_{T_1^{\lambda}}^{T_2^{\lambda}}(x_{\lambda}(r))^2v(r)dr}\\
&=&\lambda.
\end{eqnarray*}
Since $\lambda$ is an arbitrary positive constant, we obtain that $\lambda^f_1(\Sigma\setminus\Omega)=0$.

On the other hand, suppose that the assumption in (2) is satisfied. Then, for any $\lambda>\frac{\alpha^2}{4}$ there exists again a nontrivial oscillatory solution $x_\lambda(t)$ of \eqref{ODEOsc} on $[T_0,+\infty)$. Proceeding as above, we get that $\lambda_1^f(\Sigma\setminus\Omega)\leq\lambda$. The conclusion is thus straightforward since $\lambda$ is an arbitrary positive constant larger than $\frac{\alpha^2}{4}$.
\end{proof}
Adapting arguments in \cite{BMR} we obtain also the following

\begin{proposition}\label{fIndBMR}
Let $\Sigma_f$ be a complete non--compact weighted manifold and let $L$ be the Schr\"{o}dinger operator defined by
\[
\ Lu=-\Delta_fu-q(x)u,
\]
where $q(x)$ is a continuous nonnegative function on $\Sigma$. Assume that 
\begin{enumerate}
	\item[(i)] $\mathrm{vol}_f(\partial B_r(o))^{-1}\notin L^1(+\infty)$;
	\item[(ii)] $q\notin L^1(\Sigma, e^{-f}d\mathrm{vol}_{\Sigma})$.  
\end{enumerate}
Then, for an arbitrary compact subset $\Omega\subset\Sigma$ we have that the bottom of the spectrum of 
$L$ on $\Sigma\setminus\Omega$ satisfies $\lambda_1^{L}(\Sigma\setminus\Omega)<0$.
\end{proposition}

\begin{proof}
Since $\Omega$ is compact we can find a constant $T_0$ such that $\Omega\subset B_{T_0}(o)$. By Corollary 2.4 in 
\cite{BMR} we have that under our assumptions any solution $x(t)$ of 
\begin{equation}\label{ODEOscA}
\begin{cases}
(v(t)x^{\prime}(t))^{\prime}+Q(t) v(t)x(t)=0,\quad \textrm{a.e. on}\quad (T_0,+\infty),\\
x(T_0)=x_0
\end{cases}
\end{equation}
where $Q(t)=\frac{1}{v(t)}\int_{\partial B_t(o)}qe^{-f}$, is oscillatory.
Choose, as above, $v(t)=\mathrm{vol}_f(\partial B_t(o))$. Then there exists some nontrivial oscillatory solution 
$x_Q(t)$ of \eqref{ODEOscA} a.e. on $[T_0,+\infty)$, i.e., there exist $T_1^{Q}$ and $T_2^{Q}$ in $[T_0,+\infty)$ such that 
$T_1^{Q}<T_2^{Q}$ and $x_{Q}(T_1^{Q})=x_{Q}(T_2^{Q})=0$, and $x_{Q}(t)\neq 0$ for any $t\in(T_1^{Q}, T_2^{Q})$. 
Let $\varphi_Q(x)=x_Q(r(x))$ and $\Omega_Q=B_{T_2^Q}(o)\setminus B_{T_1^{Q}}(o)$. Using the co--area 
formula \eqref{co-area} we get
\begin{eqnarray*}
\int_{\Omega_Q}(|\nabla\varphi_Q|^2-q\varphi_Q^2)e^{-f}d\mathrm{vol}_{\Sigma}&=&
\int_{\Omega_Q}(x_Q^{\prime}(r)^2-qx_Q(r)^2)e^{-f}d\mathrm{vol}_\Sigma\\
&=&\int_{T_1^Q}^{T_2^Q}(x^{\prime}_{Q}(r)^2v(r)-x_Q(r)^2Q(r)v(r))dr\\
&=&-\int_{T_1^Q}^{T_2^Q}x_Q(r)((v(r)x_Q^{\prime}(r))^{\prime}+Q(r) v(r)x_Q(r))dr\\
&=&0.
\end{eqnarray*}
The conclusion follows now by strict domain monotonicity.
\end{proof}

The previous results, applied in the setting of $f$--minimal hypersurfaces allow us to obtain the following
\begin{theorem}\label{fIndFinfMin}
Let $M_f$ be a complete weighted manifold with $\overline{\mathrm{Ric}}_f\geq k>0$. Then there is no complete $f$--minimal hypersurface $\Sigma$ immersed in $M_f$ with  $Ind_f(\Sigma)<+\infty$ provided one of the following conditions holds
 \begin{enumerate}
  \item $\mathrm{vol}_f(\Sigma)=+\infty$ and $\mathrm{vol}_f(B_r(o))\leq C r^a$ for any $r\geq r_0$ and some positive constants $C$, $r_0$ and $a$;
  \item  $\mathrm{vol}_f(\partial B_r)^{-1}\notin L^{1}(+\infty)$ and $|A|\notin L^{2}(\Sigma, e^{-f}d\mathrm{vol}_{\Sigma})$.
\end{enumerate}
\end{theorem}

\begin{remark}
\rm{Observe that if $\mathrm{vol}_f(\Sigma)<+\infty$ then $\mathrm{vol}_f(\partial B_r)^{-1}\notin L^1(+\infty)$. Indeed,by the Cauchy--Schwartz inequality, we have that for all $R>0$ and $r>R$,
\[
\ \int_{R}^{r}\frac{ds}{\mathrm{vol}_f(\partial B_s)}\int_{R}^{r}\mathrm{vol}_f(\partial B_s)ds\geq (r-R)^2.
\]
Taking now the limit as $r\to\infty$ the conclusion follows. Hence, the case of finite $f$--volume is taken into account in part (2) of the theorem.
}
\end{remark}
\begin{proof}(of Theorem \ref{fIndFinfMin})
Assume that $\mathrm{Ind}_f(\Sigma)<+\infty$, $\mathrm{vol}_f(\Sigma)=+\infty$ and
\[
\mathrm{vol}_f(B_r(o))\leq C r^a 
\]
for any $r\geq r_0$ and some positive constants $C$, $r_0$ and $a$. Then, for all $r\geq r_0$,
\[
0\geq \inf_{\Sigma\backslash B_r(o)}(|A|^2+\overline{\mathrm{Ric}}_f(\nu,\nu)).
\]
This gives a contradiction in case $\overline{\mathrm{Ric}}_f\geq k>0$.

To get the proof of the remaining case we only have to observe that if 
$|A|\notin L^2(\Sigma, e^{-f}d\mathrm{vol}_{\Sigma})$ then 
$q=|A|^2+\overline{\mathrm{Ric}}_f(\nu,\nu)\notin L^1(e^{-f}d\mathrm{vol}_{\Sigma})$. 
Hence under the assumption $\mathrm{vol}_f(\partial B_r(o))^{-1}\notin L^1(+\infty)$ we can apply Proposition 
\ref{fIndBMR} to conclude the proof.
\end{proof}

\section{Finiteness results and weigthed Li--Tam theory}\label{Fin&Li-Tam}

Finiteness results for $L^2$ harmonic sections have been extensively investigated by many authors under different 
assumptions. With respect to this, we quote \cite{LW1}, \cite{LW2}, \cite{Car}, \cite{PRS-RevMatIb}, 
\cite{PRS_Progress}.

The abstract finiteness result we are going to present is an adaptation to the weighted setting of Theorem 1.1 in 
\cite{PRS-RevMatIb}; see also \cite{PRS_Progress}.

\begin{theorem}\label{f-finiteness}
Let $M_f$ be a connected, complete $m$--dimensional weighted manifold and let $E$ be a Riemannian vector bundle of rank $l$ over $M$. Denote by $\Gamma(E)$ the space of its smooth sections. Having fixed 
\[
\ a(x)\in C^{0} (M),\quad A\in\mathbb{R},\quad H\geq p
\]
satisfying the further restrictions
\begin{equation*}
p\geq A+1,\quad p>0,
\end{equation*}
let $V=V(a, f, A, p, H)\subset\Gamma (E)$ be any vector space with the following two properties.\\
\begin{enumerate}
	\item[(i)] Every $\xi\in V$ has the unique continuation property, i.e., $\xi$ is the null section whenever it vanishes on some domain.
	\item[(ii)] For any $\xi\in V$, the locally Lipschitz function $u=|\xi|$ satisfies
\begin{equation*}
\begin{cases}
u(\Delta_f u+ a(x)u)+ A|\nabla u|^2\geq 0\, \mathrm{weakly\,on}\, M\\
u\in L^{2p}(M_f).
\end{cases}
\end{equation*}
\end{enumerate}
If there exists a solution $0<\varphi\in Lip_{loc}$ of the differential inequality
\begin{equation}\label{ineq2}
\Delta_f \varphi+ Ha(x)\varphi\leq0
\end{equation}
weakly outside a compact set $K\subset M$, then
\[
\ \mathrm{dim}\, V<+\infty.
\]
\end{theorem}

\begin{proof}[Outline of the proof of Theorem \ref{f-finiteness}] We follow the arguments in Theorem 1.1 in \cite{PRS-RevMatIb}, and we refer to it for more details. Choose $R\gg1$ in such a way that $K\subset B_{R}(o)$ and, therefore, inequality \eqref{ineq2} holds in $M\setminus B_{R}(o)$. Note that, by unique continuation, the restriction map
\begin{align*}
V&\to\Gamma(\left.E\right|_{B_{R}})\\
\xi&\mapsto\left.\xi\right|_{B_{R}}
\end{align*}
is an injective homomorphism. Use the same symbol $V$ to denote the image of $V$ in $\Gamma (\left.E\right|_{B_{R}})$. Easily adapting to the weighted setting the extension, obtained in \cite[Lemma 2.1]{PRS-RevMatIb}, of a classical result by P. Li we obtain that if $T\subset V$ is any finite dimensional subspace, then there exists a (non--zero) section $\bar{\xi}\in T$ such that, setting $\bar{\psi}=|\bar{\xi}|$, it holds
\begin{equation}\label{f-LiLem}
(\mathrm{dim} T)^{\min(1, p)}\int_{B_{R}}\bar{\psi}^{2p}e^{-f}d\mathrm{vol}_{M}\leq\mathrm{vol}_f(B_{R})\min\left\{l, \mathrm{dim} T\right\}^{\min(1, p)}\sup_{B_{R}}\bar{\psi}^{2p}.
\end{equation}
Observe now that, on every sufficiently small closed ball,
\[
\ \lambda_1^{L_{H}}(B_{3\delta}(x))>0,
 \]
where $L_{H}=-\Delta_f-Ha(x)$, and therefore there exists $w>0$ solution on $B_{3\delta}(x)$ of 
\[
\ \Delta_f w+ Ha(x)w=0.
\]
Let $u\geq0$ be a locally Lipschitz, weak solution of
\begin{equation}\label{ineq1}
u(\Delta_f u+ a(x)u)+A|\nabla u|^2\geq 0.
\end{equation} 
Applying the computational Lemma 9 in \cite{Rim} with $\beta=p\geq A+1$, $\alpha=\frac{p}{H}$, setting $h=-\log w^{2\alpha}+f$, we deduce that $$v=u^{\beta}w^{-\alpha}$$ satisfies 
\begin{equation}\label{ineqh}
\Delta_h v\geq 0 \quad \mathrm{weakly\, on \,} B_{3\delta}(x)
\end{equation}
and 
\begin{equation}\label{eqnormhf}
\left\|v\right\|_{L^{2}(M_h)}=\left\|u^p\right\|_{ L^{2}(M_f)}.
\end{equation}
Since, locally, a weighted $L^2$--Sobolev inequality is always available, reasoning as in Section 2 in \cite{PRS-RevMatIb}, we are able to obtain the following local weighted $L^1$--mean value inequality for solutions $v$ of \eqref{ineqh}
\[
\ \sup_{B_{\delta}(x)}v^{2}\leq C\int_{B_{2\delta}(x)}v^2e^{-h}d\mathrm{vol}_{M}
\]
for some constant $C>0$ depending on $\left.w\right|_{\overline{B}_{2\delta}(x)}$ and the geometry of $B_{2\delta}(x)$. Recalling the definition of $v$, we deduce from the previous inequality and \eqref{eqnormhf} the following weighted $L^{p}$--mean value inequality for solutions $u$ of \eqref{ineq2}
\[
\ \sup_{B_{\delta}(x)}u^{2p}\leq C^{\prime}\int_{B_{2\delta}}u^{2p}e^{-f}d\mathrm{vol}_{M},
\]
where 
\[
\ C^{\prime}=\left(\sup_{B_{\delta}(x)}w^{\frac{p}{H}}\right)^2C.
\]
The local inequalities patch together and, in the special case of $\bar{\psi}$, give
\[
\ \sup_{B_{R}(o)}\bar{\psi}^{2p}\leq C^{\prime}\int_{B_{R+1}(o)}\bar{\psi}^{2p}e^{-f}d\mathrm{vol}_{M}.
\]
Inserting into \eqref{f-LiLem} we obtain
\begin{eqnarray}
(\mathrm{dim} T)^{\min(1, p)}&&\int_{B_{R}}\bar{\psi}^{2p}e^{-f}d\mathrm{vol}_{M}\leq C^{\prime}\mathrm{vol}_f(B_{R})\min\left\{l, \mathrm{dim} T\right\}^{\min(1, p)}\nonumber\\
&&\times\left(\int_{B_{R}}\bar{\psi}^{2p}e^{-f}d\mathrm{vol}_{M}+\int_{A(R, R+1)}\bar{\psi}^{2p}e^{-f}d\mathrm{vol}_{M}\right)\label{f-LiLem2}
\end{eqnarray}
where $A(R,R+1)$ is the annulus $B_{R+1}\setminus B_{R}$. Now considering a suitable combination of $u$ and $\varphi$, adapting the proof of Lemma 2.7 in \cite{PRS-RevMatIb} to the weighted setting in a similar way to what we just did, we obtain a weighted integral, a--priori estimate on annuli of the type
\[
\ \int_{A(R,R+1)}\bar{\psi}^{2p}e^{-f}d\mathrm{vol}_{M}\leq C^{\prime\prime}\int_{B_{R}}\bar{\psi}^{2p}e^{-f}d\mathrm{vol}_{M},
\] 
for some constant $C^{\prime\prime}$ independent of $\bar{\psi}$. From this latter and \eqref{f-LiLem2} we finally deduce
\[
\ \mathrm{dim}T\leq C^{\prime\prime\prime}\min\left\{l, \mathrm{dim}T\right\},
\]
from some $C^{\prime\prime\prime}$ depending only on the geometry of $B_{R}$. This proves that any finitely generated subspace $T$ of $V$ has dimension which is bounded by a universal constant, depending only on the rank $l$ of $E$ and on the weighted geometry of $B_{R}$. The same bound must work for the dimension of the whole $V$.
\end{proof}

In the non--weighted case there is the well--known connection, developed by P. Li and L.--F. Tam (see e.g. \cite{LT}), between $L^2$ harmonic $1$--forms, the number of non--parabolic ends, and the Morse index of the operator $-\Delta-a(x)$, where $-a(x)$ is the smallest eigenvalue of the Ricci tensor at $x$. In Theorem \ref{f-LiTam} below we shall see that an analogous relation holds in the weighted setting. This can be easily obtained with minor changes to the proofs in \cite{LT}.

Recall that an end $E$ of a weighted manifold $M_f$ with respect to a fixed compact set $D$ with smooth boundary is said to be $f$--parabolic if and only if its double is $f$--parabolic or, equivalently, if every positive $f$--superharmonic function $u$ on $E$ satisfying $\partial u/\partial \nu\geq 0$ on $\partial E$, $\nu$ being the unit outward normal to $\partial E$, is constant. Otherwise the end will be called non--$f$--parabolic. Non--$f$--parabolicity of the end $E$ can be also characterized by the existence of a positive minimal Green kernel $G_{f}$ for $\Delta_f$, satisfying Neumann boundary conditions on $\partial E$. As we said above, the following result permits to control the number of non--$f$--parabolic ends by the dimension of the space of bounded $f$--harmonic functions with finite Dirichlet weighted integral. The idea of the proof is the same as in the non--weighted case. Given two distinct $f$--parabolic ends $E_A$ and $E_B$, one can construct bounded $f$--harmonic functions $g_A$ on $M_f$ 
with finite Dirichlet weighted integral such that 
\[
\ \sup_{E_A}g_A=1\quad\quad\inf_{E_B}g_A=0,
\]
and these turn out to be linearly independent.
\begin{theorem}\label{f-LiTam}
Let $\mathcal{H}_{\mathcal{D}}^{\infty}(M_f)$ denote the space of bounded $f$--harmonic functions with finite Dirichlet weighted integral on $M_f$, and by $N(D)$ the number of non--$f$--parabolic ends of $M_f$ with respect to the relatively compact domain $D$. Then 
\[
\ N(D)\leq \mathrm{dim}\mathcal{H}_{\mathcal{D}}^{\infty}(M_f).
\]
It follows that, if $\mathcal{H}_{\mathcal{D}}^{\infty}(M_f)$ is finite dimensional, then $M_{f}$ has finitely many non--$f$--parabolic ends, whose number is bounded above by $\mathrm{dim}\mathcal{H}_{\mathcal{D}}^{\infty}(M_{f})$.
\end{theorem}

Let $\delta_f=\delta+i_{\nabla f}$, and denote with $\Delta_H^{f}=\delta_f d+ d\delta_f$ the Hodge 
$f$--Laplacian on $M_f$. 
We have that the following $f$--Weitzenbock formula for $1$--forms holds
\[
\ \frac{1}{2}\Delta_f|\omega|^2=-\left\langle \Delta_{H}^{f}\omega, \omega\right\rangle+ |D\omega|^2+\mathrm{Ric}_{f}(\omega^{\sharp}, \omega^{\sharp}).
\]
In particular, if $\omega\in\mathcal{H}_1(M_{f})=\left\{1\mathrm{-forms\,}\omega \,|\, \Delta_{H}^{f}\omega=0\right\}$, we obtain 
\begin{equation}\label{f-Weitzen}
\frac{1}{2}\Delta_f|\omega|^2=|D\omega|^2+\mathrm{Ric}_{f}(\omega^{\sharp},\omega^{\sharp}).
\end{equation}
Thus, let $\mathrm{Ric}_{f}\geq -a(x)$ for some continuous function $a(x)$, and consider the vector
space $L^{2,f}\mathcal{H}^{1}(M_{f})=\left\{\xi\in\mathcal{H}^{1}(M_f) \,|\,\, |\xi|\in L^{2}(M_f)\right\}$. Using Kato inequality, we get that, for any $\xi\in L^{2,f}\mathcal{H}^{1}(M_f)$, the locally Lipschitz function $u=|\xi|$ satisfies
\begin{equation*}
\begin{cases}
u(\Delta_f u+ a(x)u)\geq 0\,\,\, \mathrm{weakly\,on}\, M\\
\int_{M}u^{2}e^{-f}d\mathrm{vol}_M<+\infty.
\end{cases}
\end{equation*}

Moreover, note that equation $\Delta_{H}^{f}\omega=0$ is equivalent to the equation $\Delta_{H}\omega=F(x,\omega, d\omega)$ with $F$ satisfying the structural conditions of Aronszajn--Cordes; see e.g. Appendix A in \cite{PRS_Progress}. This suffices to guarantee that every $\xi\in\mathcal{H}^{1}(M_f)$ has the unique continuation property. 

We are thus in a situation where Theorem \ref{f-finiteness} can be applied. Hence, using Proposition \ref{Baptiste}, 
we obtain the following consequence of Theorem \ref{f-LiTam}. Compare also with \cite{MW1} where some related 
results are obtained.

\begin{corollary}\label{finnonfparends}
Let $M_f$ be a complete non--compact weighted manifold satisfying
\[
\ \mathrm{Ric}_f\geq -\,a(x)
\]
for some nonnegative continuous function $a(x)$, and let $L=-\Delta_f-a(x)$. Suppose furthermore that $L$ has finite 
Morse index. Then $M_f$ has at most finitely many non--$f$--parabolic ends.
\end{corollary}

\begin{proof}[Sketch of the proof of Corollary \ref{finnonfparends}]
In order to apply Theorem \ref{f-LiTam} to get the conclusion in the above corollary we have used the following fact. If $u$ is a $f$--harmonic function with finite Dirichlet weigthed integral, then its exterior differential $du$ belongs to $\mathcal{H}_1(M_f)$. Moreover $du=0$ if and only if $u\equiv const$. Hence, we have that 
\[
\ \mathrm{dim}\mathcal{H}_{\mathcal{D}}^{\infty}(M_f)\leq\mathrm{dim}\mathcal{H}_{\mathcal{D}}(M_{f})\leq \mathrm{dim} L^{2,f}\mathcal{H}^{1}(M_f)+1,
\]
where we denote by $\mathcal{H}_{\mathcal{D}}(M_f)$ the space of $f$--harmonic functions with finite Dirichlet weighted integral on $M_f$.
\end{proof}

\begin{remark}
\rm{As observed in \cite{PRS-RevMatIb}, the generality achieved in Theorem \ref{f-finiteness} permits to deal also 
with situation in which we do not have the validity of a refined Kato inequality. This is essential in our case 
since, as observed in Remark 4.2 in \cite{RV}, in general we do not have the validity of any refined Kato 
inequality for $f$--harmonic forms.}
\end{remark}

In order to deduce topological consequences from the finiteness result of the space of bounded $f$--harmonic 
functions with finite weighted Dirichlet integral on $M_f$, we need to find conditions which ensure that all ends of 
$M_f$ are non--$f$--parabolic. This can be done adapting to the weighted setting a result by 
H. D. Cao, Y. Shen, S. Zhu, \cite{CSZ}. See also \cite{BK}, where this result is proved in the more general setting 
of metric measure spaces.
\begin{lemma}\label{f-CSZ}
Let $M_f$ be a complete weighted manifold, and assume that for some $0\leq\alpha<1$, there exists a 
constant $S(\alpha)>0$ such that the weighted $L^2$--Sobolev inequality
\begin{equation}\label{f-SobL2} \left(\int_Mh^{\frac{2}{1-\alpha}}e^{-f}d\mathrm{vol}_M\right)^{1-\alpha}\leq S(\alpha)\int_M|\nabla h|^2e^{-f}d\mathrm{vol}_M
\end{equation}
holds for every smooth function compactly supported in the complement of a compact set $K$. Then every end 
$E$ of $M_f$ is either non--$f$--parabolic or it has finite $f$--volume.
\end{lemma}

\begin{remark}\label{SobL1InfVolEnds}
\rm{Suppose that $M_f$ supports a weighted $L^1$--Sobolev inequality outside a compact set $K$, namely for some 
$\alpha\in\left(1, \frac{m}{m-1}\right]$ there exists a constant $S_1(\alpha)>0$ such that
\begin{equation}\label{f-SobL1}
\left(\int_M h^{\alpha}e^{-f}d\mathrm{vol}_M\right)^{\frac{1}{\alpha}}\leq S_1(\alpha)\int_M|\nabla h|e^{-f}d\mathrm{vol}_M
\end{equation}
for every smooth function $u$ compactly supported in $M\setminus K$. Reasoning as in the non--weighted setting, see e.g. Lemma 7.15 in \cite{PRS_Progress}, one can show that every end with respect to $K$ has infinite $f$--volume and that, if $m\geq3$, \eqref{f-SobL2} holds with 
\[
\ S(\alpha)=\left(\frac{2S_{1}(\alpha)}{2-\alpha}\right)^2.
\]
As a consequence of Lemma \ref{f-CSZ}, it follows that if $M_f$ supports \eqref{f-SobL1} for some $\alpha\in\left(1, \frac{m}{m-1}\right]$ and for every smooth function $u$ compactly supported in $M\setminus K$, then every end of $M_f$ with respect to $K$ is non--$f$--parabolic.}
\end{remark} 

\section{Weighted Hessian comparison theorem}\label{Compar}

Motivated by Remark \ref{SobL1InfVolEnds}, we are interested now in proving, under suitable conditions, the validity 
of a weighted $L^1$--Sobolev inequality for an hypersurface $\Sigma$ isometrically immersed in a weighted 
manifold $M_{f}$. In the non--weighted setting, according to D. Hoffman and J. Spruck, \cite{HoSp}, minimal 
submanifolds of Cartan--Hadamard manifolds enjoy an $L^1$--Sobolev inequality. In this order of ideas, we have to 
address the issue of defining a right concept of weighted sectional curvature. 

In weighted geometry there are good concepts of Ricci and scalar curvature, namely, the Bakry--\'Emery Ricci tensor 
and the Perelman scalar curvature, defined on $M_f$ as $$P_f=R+2\Delta f-|\nabla f|^2,$$
where $R$ is the scalar curvature of $M$. On the other hand, as far as we know, there is no concept of 
sectional curvature associated to a weighted manifold and, in general, to a measure. As observed in \cite{CGY}, 
both $\mathrm{Ric}_f$ and $P_f$, can be viewed as the infinite--dimensional limit of their conformally invariant 
counterparts. Trying to carry out the same process for the full curvature tensor one easily realizes that, 
``letting the dimension go to infinity", the conformally invariant counterpart of the Riemann tensor recovers 
the Riemann tensor itself. This is not so surprising from the viewpoint of sectional curvature, since sectional 
curvature only takes into account two--dimensional subspaces, and hence the dimension plays no role in defining this concept. This informal discussion suggests that a good concept of sectional curvature in weighted geometry should be the sectional curvature itself. This assertion is supported by the following comparison theorem.

\begin{theorem}\label{HessfCompTh}
Let $M_f^{m}$ be a complete weighted $m$--dimensional manifold. Having fixed a reference point $o\in M$, let $r(x)=\mathrm{dist}_M(x, o)$ and let $D_o=M\setminus \mathrm{cut}(o)$ be the domain of the normal geodesic coordinates centered at $o$. Given a smooth even function $G$ on $\mathbb{R}$, let $h$ be the solution of the Cauchy problem
\begin{equation}\label{CauchyMod}
\begin{cases}
h^{\prime\prime}-G h=0\\
h(0)=0,\quad h^{\prime}(0)=1
\end{cases}
\end{equation}
and let $I=\left[0,r_0\right)\subseteq \left[0, +\infty\right)$ be the maximal interval where $h$ is positive. Suppose that the radial sectional curvature of $M$, that is the sectional curvature of $2$--planes containing $\nabla r$, satisfies
\begin{equation}\label{SectComp}
\mathrm{Sect}_{\mathrm{rad}}\geq - G(r(x))\quad (\mathrm{resp.}\, \leq)
\end{equation}
on $B_{r_0}(o)$ and, furthermore, assume that
\begin{equation}\label{Contrdrf}
\eta(r)=\left\langle \nabla r, \nabla f\right\rangle\geq -\theta(r)\quad (\mathrm{resp.}\, \leq)
\end{equation}
for some $\theta\in C^{0}\left(\left[0, +\infty\right)\right)$, and $\eta(s)=o(1)$ as $s\to0^{+}$. Let
\begin{equation*}
\mathrm{Hess}_f(r):=\mathrm{Hess}(r)-\frac{1}{m}\left\langle \nabla f, \nabla r\right\rangle\left\langle \cdot\,,\,\cdot\right\rangle
\end{equation*}
then
\begin{equation}\label{HessfComp}
\mathrm{Hess}_f(r)\leq\frac{h^{\prime}}{h}\left\{\left\langle \cdot\,,\,\cdot\right\rangle-dr\otimes dr\right\}+\frac{1}{m}\theta(r)\left\langle \cdot\,,\,\cdot\right\rangle\quad (\mathrm{resp.}\, \geq).
\end{equation}
\end{theorem}

\begin{remark}
\rm{Note that tracing \eqref{HessfComp}, we recover corresponding estimates for $\Delta_f r$. These are consistent with comparison results for weighted manifolds with $\mathrm{Ric}_f(\nabla r, \nabla r)$ bounded from below by $-(m-1)G(r)$ and $f$ satisfying \eqref{Contrdrf} for some non--decreasing function $\theta\in C^{0}([0,+\infty))$, see Theorem 3.1 in \cite{PRRS}.
}
\end{remark}
\begin{proof}
Observe, first of all, that $\mathrm{Hess}(r)(\nabla r, X)=0$ for all $X\in T_xM$ and $x\in D_o\setminus \left\{o\right\}$. Next, since $\mathrm{Hess}_f(r)$ is symmetric, $T_x M$ has an orthonormal basis consisting of eigenvectors of $\mathrm{Hess}_f(r)$. Denoting by $\lambda_{\mathrm{max}}(x)$ and $\lambda_{\mathrm{min}}(x)$, respectively, the greatest and the smallest eigenvalues of $\mathrm{Hess}_f(r)$ in the orthogonal complement of $\nabla r(x)$, the theorem amounts to showing that on $D_o\setminus\left\{o\right\}\cap B_{r_o}(o)$ 
\begin{enumerate}
	\item[(i)] if \eqref{SectComp} and \eqref{Contrdrf} hold with $\geq$, then $\lambda_{\mathrm{max}}\leq\frac{h^{\prime}}{h}(r(x))+\frac{1}{m}\theta(r(x))$; 
	\item[(ii)] if \eqref{SectComp} and \eqref{Contrdrf} hold with $\leq$, then $\lambda_{\mathrm{min}}\geq\frac{h^{\prime}}{h}(r(x))+\frac{1}{m}\theta(r(x))$.
\end{enumerate}
Let us prove case (ii). The argument in case (i) is completely similar. Let $x\in D_o\setminus\left\{o\right\}$, and let $\gamma$ be the minimizing geodesic joining $o$ to $x$. We claim that $\psi=\left(\lambda_{\mathrm{min}}+\frac{\eta}{m}\right)\circ\gamma$ satisfies
\begin{equation}\label{RicIneq}
\begin{cases}
\psi^{\prime}+\psi^{2}\geq G\\
\psi(s)=\frac{1}{s}+o(1)\quad\mathrm{as}\quad s\to0^{+}
\end{cases}
\end{equation}
Since $\phi=\frac{h^{\prime}}{h}$ satisfies
\begin{equation}\label{Sturm}
\begin{cases}
\phi^{\prime}+\phi^2=G\quad\mathrm{on}\quad(0,r_0)\\
\phi(s)=\frac{1}{s}+o(1)\quad\mathrm{as}\quad s\to0^{+},
\end{cases}
\end{equation}
the required conclusion follows at once from Corollary 2.2 in \cite{PRS_Progress}. To prove the claim we proceed as follows. Let $\gamma$ be a minimizing geodesic joining $o$ to $\gamma(s_{0})=x\in D_{o}\setminus\left\{o\right\}$. For every unit vector $Y\in T_{x}M$ such that $Y\bot\dot{\gamma}(s_0)$, define a vector field $Y\bot\dot{\gamma}$, by parallel translation along $\gamma$. Since as noted above $\mathrm{hess}(r)(\nabla r)=\nabla_{\nabla r}\nabla r$=0, we compute, as in \cite{PRS_Progress}, 
\begin{equation}\label{eqcomp1}
\frac{d}{ds}(\mathrm{Hess}(r)(\gamma)(Y, Y))+\left\langle \mathrm{hess}(r)(\gamma)(Y), \mathrm{hess}(r)(\gamma)(Y)\right\rangle=-\mathrm{Sect}_{\gamma}(Y\wedge\dot{\gamma}).
\end{equation}
Moreover, we have that
\begin{eqnarray}
\frac{d}{ds}(\mathrm{Hess}_{f}(r)(\gamma)(Y,Y))&=&\frac{d}{ds}(\mathrm{Hess}(r)(\gamma)(Y, Y))-\frac{1}{m}\frac{d}{ds}\left\langle \nabla r\circ\gamma,\nabla f\circ\gamma\right\rangle\nonumber\\
&=&\frac{d}{ds}(\mathrm{Hess}(r)(\gamma)(Y, Y))-\frac{1}{m}\frac{d}{ds}\eta\circ\gamma\label{eqcomp2}
\end{eqnarray}
and letting
\begin{equation*}\label{eqcomp3}
\mathrm{hess}_{f}(r)(\gamma)(Y)=\mathrm{hess}(r)(\gamma)(Y)-\frac{1}{m}\left(\eta\circ\gamma\right) Y
\end{equation*}
we have that
\begin{eqnarray}
\left\langle \mathrm{hess}_f(r)(\gamma)(Y), \mathrm{hess}_{f}(r)(\gamma)(Y)\right\rangle&=&\left\langle \mathrm{hess}(r)(\gamma)(Y), \mathrm{hess}(r)(\gamma)(Y)\right\rangle\nonumber\\&&-\frac{2}{m}\mathrm{Hess}(r)(\gamma)(Y, Y)\left(\eta\circ\gamma\right)\label{eqcomp4}\\&&+\frac{1}{m^2}\left(\eta\circ\gamma\right)^2.\nonumber
\end{eqnarray}
Hence, by \eqref{eqcomp1}, \eqref{eqcomp2}, \eqref{eqcomp4}, and the lower bound in \eqref{SectComp}, we get that along $\gamma$
\begin{eqnarray}
\frac{d}{ds}(\mathrm{Hess}_f(r)(Y, Y))+ |\mathrm{hess}_f(r)(Y)|^2&\geq& G(r)-\frac{1}{m}\frac{d}{ds}\left(\eta\circ\gamma\right)\nonumber\\&&-\frac{2}{m}\mathrm{Hess}(r)(Y, Y)\left(\eta\circ\gamma\right)\label{eqcomp5}\\&&+\frac{1}{m^2}\left(\eta\circ\gamma\right)^2.\nonumber
\end{eqnarray}
Note that for any unit vector field $X\bot\nabla r$
\[
\ \mathrm{Hess}_{f}(r)(\gamma)(X, X)\geq \lambda_{\mathrm{min}.}
\]
Thus, if $Y$ is choosen so that, at $s_0$
\[
\ \mathrm{Hess}_{f}(r)(\gamma)(Y, Y)=\lambda_{\mathrm{min}}(\gamma(s_0)),
\]
then the function $\mathrm{Hess}_{f}(r)(\gamma)(Y, Y)-\lambda_{\mathrm{min}}\circ\gamma$ attains its minimum at $s=s_0$ and, if at this point $\lambda_{\mathrm{min}}$ is differentiable, then its derivative vanishes:
\begin{equation*}
\left.\frac{d}{ds}\right|_{s_0}\mathrm{Hess}_{f}(r)(\gamma)(Y, Y)-\left.\frac{d}{ds}\right|_{s_0}\lambda_{\mathrm{min}}\circ\gamma=0.
\end{equation*}
Whence, using \eqref{eqcomp5}, we obtain that, at $s_0$,
\begin{eqnarray*}
\frac{d}{ds}(\lambda_{\mathrm{min}}\circ\gamma)+(\lambda_{\mathrm{min}}\circ\gamma)^2&\geq& G(r)-\frac{1}{m}\frac{d}{ds}\left(\eta\circ\gamma\right)\\&&-\frac{2}{m}\mathrm{Hess}(r)(Y, Y)\left(\eta\circ\gamma\right)+\frac{1}{m^2}\left(\eta\circ\gamma\right)^2\\&=&G(r)-\frac{1}{m}\frac{d}{ds}\left(\eta\circ\gamma\right)\\&&-\frac{2}{m}\mathrm{Hess}_f(r)(Y, Y)\left(\eta\circ\gamma\right)-\frac{1}{m^2}\left(\eta\circ\gamma\right)^2\\
&=&G(r)-\frac{d}{ds}\frac{\eta\circ\gamma}{m}-2(\lambda_{\mathrm{min}}\circ\gamma)\frac{\eta\circ\gamma}{m}-\frac{\left(\eta\circ\gamma\right)^2}{m^2}.
\end{eqnarray*}
Letting now $\psi=\left(\lambda_{\mathrm{min}}+\frac{\eta}{m}\right)\circ\gamma$ we get the desired differential inequality \eqref{RicIneq}. The asymptotic behaviour 
\[
\ \psi(s)=\frac{1}{s}+o(1)\quad\mathrm{as}\quad s\to0^+
\]
follows from our assumptions on $\eta$ and the fact that
\[
\ \mathrm{Hess}(r)=\frac{1}{r}(\left\langle \cdot\,,\,\cdot\right\rangle-dr\otimes dr)+o(1)\quad\mathrm{as}\quad r\to0^{+}.
\]
\end{proof}

\section{A Sobolev inequality in the weighted setting}\label{SobWeight}
In this section we prove a general weighted $L^1$--Sobolev inequality for submanifolds $\Sigma^m$ of a weighted manifold $M_{f}^{m+1}$, satisfying some restrictions on $f$ and on the sectional curvature of $M$. The proof is inspired by the papers of J. H Michael, and L. M. Simon, \cite{MiSi}, and of D. Hoffman and J. Spruck, \cite{HoSp}. Recall that with $\overline{(\cdot)}$ we refer to quantities in the ambient space.
\begin{theorem}\label{ThWeightSobL1}
Let $\Sigma^m\to M_f^{m+1}$ be an isometric immersion. Assume that $\overline{\mathrm{Sect}}\leq 0$ and suppose that there exists a positive constant $c_{m}$ such that
\begin{equation}\label{smallweightgeom}
\limsup_{\rho\to0^{+}}\frac{\mathrm{vol}_{f}(S_{\rho}(\xi))}{\rho^{m}}\geq c_{m},
\end{equation} 
for almost all $\xi\in M_f$, where we are using the notation
\[
\ S_{\rho}(\xi)=\left\{x\in \Sigma \,|\, {}^{M}\mathrm{dist}(x, \xi)\leq \rho\right\}.
\]
Let $h$ be a non--negative compactly supported $C^{1}$ function on $\Sigma$. Then
\begin{equation}\label{WeightSobL1}
\left[\int_{\Sigma}h^{\frac{m}{m-1}}e^{-f}d\mathrm{vol}_{\Sigma}\right]^{\frac{m-1}{m}}\leq C\left[\int_{\Sigma}|\nabla h|+h\left(\left|H_{f}\right|+|\overline{\nabla }f|\right)e^{-f}d\mathrm{vol}_{\Sigma}\right].
\end{equation}
\end{theorem}
\begin{remark}
\rm{Note that for every isometric immersion $\Sigma^m\to M^{m+1}$ we have that
\[
\ \limsup_{\rho\to0^{+}}\frac{\mathrm{vol}(S_{\rho}(\xi))}{\rho^m}\geq \omega_{m}
\]
for almost all $\xi\in M$, with $\omega_{m}$ the volume of the unit ball in $\mathbb{R}^{m}$. Hence condition \eqref{smallweightgeom} is satisfied if we assume that $f<f^{*}<+\infty$.}
\end{remark}

\begin{remark}
\rm{Theorem \ref{ThWeightSobL1} has a companion weighted isoperimetric inequality. In this regard, we mention that the isoperimetric problem in Riemannian manifolds with density (and in particular in the Gaussian space) is a recent and very active field of research; see e.g. \cite{Mo-notices}, \cite{RCBM}, \cite{MoPr}, \cite{Mil}.}
\end{remark}
Let $\Sigma^m\to M_f^{m+1}$ be an isometric immersion as in Theorem \ref{ThWeightSobL1}, let $X=r\overline{\nabla}r$, $r$ distance function on $M$, and $\left\{E_1, \ldots, E_{m}\right\}$ be a local orthonormal frame on $\Sigma$. Since
\[
\ ^{\Sigma}\mathrm{div} X=\sum_{i=1}^{m}\left\langle \overline{\nabla}_{E_i}(r\overline{\nabla}r), E_{i}\right\rangle=r\sum_{i=1}^{m}\overline{\mathrm{Hess}}(r)(E_{i}, E_{i})+\sum_{i=1}^{m}\left\langle \overline{\nabla}r, E_{i}\right\rangle^2
\]
we obtain, by the classical Hessian comparison theorem, that
\begin{eqnarray}
{}^{\Sigma}\mathrm{div}X-\left\langle \overline{\nabla}f, X\right\rangle&=&r\sum_{i=1}^{m}\overline{\mathrm{Hess}}(r)(E_{i}, E_{i})+\sum_{i=1}^{m}\left\langle \overline{\nabla}r, E_{i}\right\rangle^2\label{EqSob1}\\
&&-r\left\langle \overline{\nabla}f,\overline{\nabla}r\right\rangle\nonumber\\
&\geq&m-\sum_{i=1}^{m}\left\langle \overline{\nabla}r, E_{i}\right\rangle^2+\sum_{i=1}^{m}\left\langle \overline{\nabla}r, E_{i}\right\rangle^2-r\left\langle \overline{\nabla}f,\overline{\nabla}r\right\rangle\nonumber\\
&\geq&m-r|\overline{\nabla}f|.\nonumber
\end{eqnarray}
Let $\lambda$ be a non--decreasing $C^{1}$ function on $\mathbb{R}$ with $\lambda(t)=0$ for $t\leq0$. Let $0\leq h\in C_c^{1}(\Sigma)$. For $\xi\in\Sigma$, let $r(x)$ be the distance function from the point $\xi$ on $M$. Then we define the following quantities
\begin{eqnarray*}
\phi_{\xi}(\rho)&=&\int_{\Sigma}\lambda(\rho-r(x))h(x)e^{-f}d\mathrm{vol}_{\Sigma};\\
\psi_{\xi}(\rho)&=&\int_{\Sigma}\lambda(\rho-r(x))\left(|\nabla h(x)|+h(x)|H_f(x)|\right)e^{-f}d\mathrm{vol}_{\Sigma};\\
\mu_{\xi}(\rho)&=&\int_{\Sigma}\lambda(\rho-r(x))\left(|\overline{\nabla}f|(x)\,h(x)\right)e^{-f}d\mathrm{vol}_{\Sigma};\\
\overline{\phi}_{\xi}(\rho)&=&\int_{S_{\rho}(\xi)}h(x)e^{-f}d\mathrm{vol}_{\Sigma};\\
\overline{\psi}_{\xi}(\rho)&=&\int_{S_{\rho}(\xi)}\left(|\nabla h(x)|+h(x)|H_f(x)|\right)e^{-f}d\mathrm{vol}_{\Sigma};\\
\overline{\mu}_{\xi}(\rho)&=&\int_{S_{\rho}(\xi)}\left(|\overline{\nabla}f|(x)h(x)\right)e^{-f}d\mathrm{vol}_{\Sigma};
\end{eqnarray*}
We now prove two lemmas which generalize Lemmas 4.1 and 4.2 in \cite{HoSp}. The first one relates the growth of $\phi_{\xi}(\rho)$ to $\psi_{\xi}(\rho)$ and $\mu_{\xi}(\rho)$.
\begin{lemma}
Let $\Sigma^m\to M_{f}^{m+1}$ be an isometric immersion. Assume that  $\overline{\mathrm{Sect}}\leq0$. Then
\begin{equation}\label{EqLemma1}
-\frac{d}{d\rho}\left(\rho^{-m}\phi_{\xi}(\rho)\right)\leq\rho^{-m}\left[\psi_{\xi}(\rho)+\mu_{\xi}(\rho)\right].
\end{equation}
\end{lemma}

\begin{proof}
Let $X$ be the radial vector field centered at $\xi$ and let $(\cdot)^{T}$ denote the projection on the tangent bundle of $\Sigma$. Since
\begin{eqnarray*} ^{\Sigma}\mathrm{div}_{f}(\lambda(\rho-r)hX^{T})&=&\lambda(\rho-r)h\,^{\Sigma}\mathrm{div}_{f}(X^{T})-\lambda^{\prime}(\rho-r)h\left\langle X^T, \nabla r\right\rangle\\&&+\lambda(\rho-r)\left\langle \nabla h, X^{T}\right\rangle,
\end{eqnarray*}
and
\begin{eqnarray*}
^{\Sigma}\mathrm{div}_{f}(X^T)&=&\sum_{i=1}^{m}\left\langle \overline{\nabla}_{E_{i}}X^{T}, E_{i}\right\rangle-\left\langle \nabla f, X^{T}\right\rangle\\
&=&\sum_{i=1}^{m}\left\langle \overline{\nabla}_{E_i}X, E_{i}\right\rangle-\left\langle X, \nu\right\rangle\sum_{i=1}^{m}\left\langle \overline{\nabla}_{E_i}\nu, E_{i}\right\rangle\\
&&-\left\langle \overline{\nabla}f, X\right\rangle+\left\langle \overline{\nabla}f, \nu\right\rangle\left\langle X, \nu\right\rangle\\
&=&^{\Sigma}\mathrm{div}X-\left\langle \overline{\nabla}f, X\right\rangle+\left\langle X, \nu\right\rangle\left(H+\left\langle \overline{\nabla}f, \nu\right\rangle\right)\\
&=&^{\Sigma}\mathrm{div}X-\left\langle \overline{\nabla}f, X\right\rangle+\left\langle X, \nu\right\rangle H_{f},
\end{eqnarray*}
we obtain that
\begin{eqnarray}
^{\Sigma}\mathrm{div}_{f}\left(\lambda(\rho-r)hX^{T}\right)&=&\lambda(\rho-r)h\left(^{\Sigma}\mathrm{div}X-\left\langle \overline{\nabla}f, X\right\rangle\right)\nonumber\\
&&+\lambda(\rho-r)h\left\langle X, \nu\right\rangle H_{f}-\lambda^{\prime}(\rho-r)h\left\langle X^{T},\nabla r\right\rangle\label{Eq1ProofLem1}\\
&&+\lambda(\rho-r)\left\langle \nabla h, X^{T}\right\rangle.\nonumber
\end{eqnarray}
Since $|\nabla r|=\left|(\overline{\nabla} r)^{T}\right|\leq|\overline{\nabla}r|=1$ and $\lambda(\rho-r)=\lambda^{\prime}(\rho-r)=0$ for $r\geq\rho$, integrating \eqref{Eq1ProofLem1} over $\Sigma$ with respect to the weighted volume measure and using the $f$--divergence theorem, we get that
\begin{eqnarray*}
\int_{\Sigma}\lambda(\rho-r)h\left(^{\Sigma}\mathrm{div}X-\left\langle X, \overline{\nabla}f\right\rangle\right)e^{-f}d\mathrm{vol}_{\Sigma}&=&\int_{\Sigma}\lambda^{\prime}(\rho-r)h\left\langle X^T, \nabla r\right\rangle e^{-f}d\mathrm{vol}_{\Sigma}\\
&&-\int_{\Sigma}\lambda(\rho-r)h\,H_{f}\,r\left\langle \overline{\nabla}r, \nu\right\rangle e^{-f}d\mathrm{vol}_{\Sigma}\\
&&-\int_{\Sigma}\lambda(\rho-r)r\left\langle \nabla r, \nabla h\right\rangle e^{-f}d\mathrm{vol}_{\Sigma}\\
&\leq&\int_{\Sigma}r\lambda^{\prime}(\rho-r)|h|e^{-f}d\mathrm{vol}_{\Sigma}\\
&&+\int_{\Sigma}r\,\lambda(\rho-r)|h|\,|H_{f}|e^{-f}d\mathrm{vol}_{\Sigma}\\
&&+\int_{\Sigma}r\lambda(\rho-r)|\nabla h|e^{-f}d\mathrm{vol}_{\Sigma}\\
&\leq&\rho\phi_{\xi}^{\prime}(\rho)+\rho\psi_{\xi}(\rho).
\end{eqnarray*}
Hence, by \eqref{EqSob1} we have that
\begin{eqnarray*}
\rho\phi_{\xi}^{\prime}(\rho)+\rho\psi_{\xi}(\rho)&\geq&\int_{\Sigma}(m-r|\overline{\nabla}f|)\lambda(\rho-r)he^{-f}d\mathrm{vol}_{\Sigma}\\
&=&m\phi_{\xi}(\rho)-\int_{\Sigma}\lambda(\rho-r)r|\overline{\nabla}f|he^{-f}d\mathrm{vol}_{\Sigma}\\
&\geq&m\phi_{\xi}(\rho)-\rho\int_{\Sigma}\lambda(\rho-r)|\overline{\nabla}f|he^{-f}d\mathrm{vol}_{\Sigma},
\end{eqnarray*}
that is,
\begin{equation}\label{ODELem1}
m\phi_{\xi}(\rho)-\rho\mu_{\xi}(\rho)\leq\rho\phi_{\xi}^{\prime}(\rho)+\rho\psi_{\xi}(\rho),
\end{equation}
proving \eqref{EqLemma1}.
\end{proof}

\begin{lemma}\label{Lemma2}
Let $\xi\in \Sigma$ be such that $h(\xi)\geq1$. Let $\alpha, t$ satisfy $0<\alpha<1\leq t$, and suppose that there exists a constant $c_m$ such that \eqref{smallweightgeom} holds. Set
\[
\ \rho_{0}=\frac{1}{1-\alpha}\left[c_m^{-1}\int_{\Sigma}he^{-f}d\mathrm{vol}_{\Sigma}\right]^{\frac{1}{m}}.
\]
Then there exists $\rho$, $0<\rho<\rho_0$, such that
\[
\ \overline{\phi}_{\xi}(t\rho)\leq\alpha^{-1}t^{m-1}\rho_{0}\left[\overline{\psi}_{\xi}(\rho)+\overline{\mu}_{\xi}(\rho)\right].
\]
\end{lemma}
\begin{proof}
Integrating \eqref{EqLemma1} on $(\sigma, \rho_{0})$, $\sigma\in(0, \rho_{0})$, we have that
\[
\ \sigma^{-m}\phi_{\xi}(\sigma)\leq\rho_{0}^{-m}\phi_{\xi}(\rho_{0})+\int_{0}^{\rho_{0}}\rho^{-m}\psi_{\xi}(\rho)d\rho+\int_{0}^{\rho_{0}}\rho^{-m}\mu_{\xi}(\rho)d\rho.
\]
Take $\varepsilon\in(0,\sigma)$ and choose $\lambda$ such that $\lambda(t)=1$ for $t\geq\epsilon$. Then
\[
\ \sigma^{-m}\overline{\phi}_{\xi}(\sigma-\epsilon)\leq\rho_{0}^{-m}\overline{\phi}_{\xi}(\rho_{0})+\int_{0}^{\rho_{0}}\rho^{-m}\overline{\psi}_{\xi}(\rho)d\rho+\int_{0}^{\rho_{0}}\rho^{-m}\overline{\mu}_{\xi}(\rho)d\rho.
\]
Hence, since $\sigma, \varepsilon$ are arbitrary,
\begin{equation*}\label{Eq1Lem2}
\sup_{\sigma\in(0,\rho_{0})}\sigma^{-m}\overline{\phi}_{\xi}(\sigma)\leq\rho_{0}^{-m}\overline{\phi}_{\xi}(\rho_{0})+\int_{0}^{\rho_{0}}\rho^{-m}\overline{\psi}_{\xi}(\rho)d\rho+\int_{0}^{\rho_{0}}\rho^{-m}\overline{\mu}_{\xi}(\rho)d\rho.
\end{equation*}
By contradiction, assume that for all $\rho\in(0, \rho_{0})$,
\[
\ \overline{\psi}_{\xi}(\rho)+\overline{\mu}_{\xi}(\rho)<\alpha t^{-(m-1)}\rho_{0}^{-1}\overline{\phi}_{\xi}(t\rho).
\]
Then
\begin{eqnarray*}
\int_{0}^{\rho_{0}}\rho^{-m}\overline{\psi}_{\xi}(\rho)d\rho&+&\int_{0}^{\rho_{0}}\rho^{-m}\overline{\mu}_{\xi}(\rho)d\rho\\
&<&\alpha\rho_{0}^{-1}\int_{0}^{\rho_{0}}t^{-(m-1)}\overline{\phi}_{\xi}(t\rho)\rho^{-m}d\rho\\
&=&\alpha\rho_{0}^{-1}\int_{0}^{t\rho_{0}}s^{-m}\overline{\phi}_{\xi}(s)ds\\
&\leq&\alpha\rho_{0}^{-1}\left[\int_{0}^{\rho_{0}}s^{-m}\overline{\phi}_{\xi}(s)ds+\int_{\rho_{0}}^{+\infty}s^{-m}\overline{\phi}_{\xi}(s)ds\right]\\
&\leq&\alpha\sup_{\sigma\in(0, \rho_{0})}\sigma^{-m}\overline{\phi}_{\xi}(\sigma)+\alpha\rho_{0}^{-m}(m-1)^{-1}\int_{\Sigma}he^{-f}d\mathrm{vol}_{\Sigma}.
\end{eqnarray*}
Thus we get that
\[
\ (1-\alpha)\sup_{\sigma\in(0,\rho_{0})}\sigma^{-m}\overline{\phi}_{\xi}(\sigma)<\rho_{0}^{-m}\int_{\Sigma}he^{-f}d\mathrm{vol}_{\Sigma}\left[1+\alpha(m-1)^{-1}\right]. 
\]
Using \eqref{smallweightgeom}, this gives a contradiction.
\end{proof}

\begin{proof}[Proof of Theorem \ref{ThWeightSobL1}]
We follow the argument in \cite{MiSi}, \cite{HoSp}.\\ Let $A=\left\{\xi\in \Sigma\,|\,h(\xi)\geq1\right\}$. Set $\rho_{i}=\beta^{i}\rho_{0}$, where $\frac{2}{t}<\beta<1$, $t>2$. Define
\[
\ A_{i}=\left\{\xi\in A \,|\,\overline{\phi}_{\xi}(t\rho)\leq\alpha^{-1}t^{m-1}\rho_{0}\left[\overline{\psi}_{\xi}(\rho)+\overline{\mu}_{\xi}(\rho)\right]\, \mathrm{for\, some}\, \rho\in\left(\rho_{i+1}, \rho_{i}\right)\right\}.
\]
It follows from Lemma \ref{Lemma2} that $A=\bigcup_{i=0}^{\infty}A_{i}$. Next, define inductively a sequence $F_{0}, F_{1}, \ldots$ of subsets of $A$ as follows:
\begin{enumerate}
	\item $F_{0}=\emptyset$;
	\item Let $k\geq1$ and assume $F_{0}, F_{1},\ldots, F_{k-1}$ have been defined. Let $B_{k}=A_{k}\setminus\bigcup_{i=0}^{k-1}\bigcup_{\xi\in F_{i}}S_{\beta t\rho_{i}}(\xi)$.
\end{enumerate}
If $B_{k}=\emptyset$, then put $F_{k}=\emptyset$. If $B_{k}\neq\emptyset$, define $F_{k}$ to be a finite subset of $B_{k}$ such that $B_{k}\subset\bigcup_{\xi\in F_{k}}S_{\beta t \rho_{k}}(\xi)$ and the sets $S_{\rho_{k}}(\xi)$ are pairwise disjoint. Then one checks that the following properties hold:
\begin{enumerate}
	\item[(a)]$F_{i}\subset A_{i}$;
	\item[(b)]$A\subset\bigcup_{i=1}^{\infty}\bigcup_{\xi\in F_{i}}S_{\beta t\rho_{i}}(\xi)$;
	\item[(c)] For all $i$, $\left\{S_{\rho_{i}}(\xi)\right\}_{\xi\in F_{i}}$ is a collection of pairwise disjoint sets.
\end{enumerate}
Let $\xi\in F_{i}$. Then, by property (a) we have that, for some $\rho\in(\beta\rho_{i}, \rho_{i})$,
\[
\ \overline{\phi}_{\xi}(t\rho)\leq \alpha^{-1}t^{m-1}\rho_{0}\left[\overline{\psi}_{\xi}(\rho)+\overline{\mu}_{\xi}(\rho)\right].
\]
Thus, since $\theta\leq0$, $\overline{\mu}_{\xi}(\rho)$ is non--decreasing and hence
\begin{eqnarray*}
\overline{\phi}_{\xi}(\beta t\rho_{i})&\leq&\overline{\phi}_{\xi}(t\rho)\leq\alpha^{-1}t^{m-1}\rho_{0}\left[\overline{\psi}_{\xi}(\rho)+\overline{\mu}_{\xi}(\rho)\right]\\
&\leq&\alpha^{-1}t^{m-1}\rho_{0}\left[\overline{\psi}_{\xi}(\rho_{i})+\overline{\mu}_{\xi}(\rho_{i})\right]. 
\end{eqnarray*}
Summing over all $\xi\in F_{i}$ and $i$ and using properties (b) and (c) defining $\Sigma_s=\left\{\xi\in\Sigma\,|\,h(\xi)\geq s\right\}$, we get that
\begin{eqnarray}
\mathrm{vol}_{f}(\Sigma_{1})&=&\sum_{i=1}^{\infty}\sum_{\xi\in F_{i}}\mathrm{vol}_{f}\left(S_{\beta t\rho_{i}}(\xi)\cap\Sigma\right)\leq \sum_{i=1}^{\infty}\sum_{\xi\in F_{i}}\overline{\phi}_{\xi}(\beta t\rho_{i})\nonumber\\
&\leq&\sum_{i=1}^{\infty}\sum_{\xi\in F_{i}}\left[\alpha^{-1}t^{m-1}\rho_{0}\left(\overline{\psi}_{\xi}(\rho_{i})+\overline{\mu}_{\xi}(\rho_{i}))\right)\right]\nonumber\\
&\leq&\alpha^{-1}t^{m-1}\rho_{0}\left[\int_{\Sigma}\left(|\nabla h|+h|H_{f}|\right)e^{-f}d\mathrm{vol}_{\Sigma}+\int_{\Sigma}|\overline{\nabla} f|he^{-f}d\mathrm{vol}_{\Sigma}\right]\nonumber.
\end{eqnarray}
Now let $s,\varepsilon>0$ be arbitrary and let $\lambda\in C^{1}(\mathbb{R})$ be non--decreasing and such that $\lambda(t)=0$ for $t\leq -\varepsilon$ and $\lambda(t)=1$ for $t\geq0$. Since we have also that
\[
\ \Sigma_{s}=\left\{\xi\in\Sigma\,|\, \lambda(h(x)-s)\geq1\right\},
\]
replacing $h$ by $\lambda(h-s)$ in the last computation, one obtains
\begin{eqnarray}
\mathrm{vol}_{f}(\Sigma_{s})&\leq&\frac{\alpha^{-1}}{1-\alpha}t^{m-1}\left[c_{m}^{-1}\int_{\Sigma}\lambda(h-s)e^{-f}d\mathrm{vol}_{\Sigma}\right]^{\frac{1}{m}}\label{Eq2ProofTh}\\
&&\times\left[\int_{\Sigma}\lambda^{\prime}(h-s)|\nabla h|+\lambda(h-s)\left[|H_{f}|+|\overline{\nabla}f|\right]e^{-f}d\mathrm{vol}_{\Sigma}\right].\nonumber
\end{eqnarray}
Multiplying both sides of \eqref{Eq2ProofTh} by $s^{\frac{1}{m-1}}$, using the fact that $\lambda(h-s)=0$ for $s\geq h+\varepsilon$, and letting $c=\alpha^{-1}(1-\alpha)^{-1}t^{m-1}c_{m}^{-\frac{1}{m}}$, we obtain
\begin{eqnarray*}
s^{\frac{1}{m-1}}\mathrm{vol}_{f}(\Sigma_s)&\leq& c\left[\int_{\Sigma}\left(h+\varepsilon\right)^{\frac{m}{m-1}}e^{-f}d\mathrm{vol}_{\Sigma}\right]^{\frac{1}{m}}\\
&&\times\left[\int_{\Sigma}\lambda^{\prime}(h-s)|\nabla h|+\lambda(h-s)\left(|H_{f}|+|\overline{\nabla}f|\right)e^{-f}d\mathrm{vol}_{\Sigma}\right]. 
\end{eqnarray*}
Finally, we integrate over $(0,+\infty)$ with respect to $s$ and let $\varepsilon\to 0$. The desired inequality \eqref{WeightSobL1} follows noting that
\begin{eqnarray*}
\int_{0}^{+\infty}s^{\frac{1}{m-1}}\mathrm{vol}_{f}(\Sigma_s)dt&=&\int_{0}^{+\infty}s^{\frac{1}{m-1}}\left(\int_{\Sigma_{s}}e^{-f}d\mathrm{vol}_{\Sigma}\right)ds\\
&=&\frac{m-1}{m}\int_{\Sigma}\left[\int_{0}^{h}\frac{m}{m-1}s^{\frac{1}{m-1}}ds\right]e^{-f}d\mathrm{vol}_{\Sigma}\\
&=&\frac{m-1}{m}\int_{\Sigma}h^{\frac{m}{m-1}}e^{-f}d\mathrm{vol}_{\Sigma},\\
\int_{0}^{+\infty}\lambda(h-s)ds&\leq& h+\varepsilon,\\
\int_{0}^{+\infty}\lambda^{\prime}(h-s)ds&\leq&1.
\end{eqnarray*}
\end{proof}

\section{Topological results}\label{TopoRes}

By Gauss equation it is not difficult to see that, given an $f$--minimal hypersurface
$x:\Sigma^m\rightarrow M_f^{m+1}$, the Bakry--\'Emery Ricci tensor of $\Sigma$ satisfies
\begin{equation}\label{Ricfrelation}
\mathrm{Ric}_f(X,X)=\overline{\mathrm{Ric}}_f(X,X)-\overline{\mathrm{Sect}}(X,\nu)|X|^2-\langle A^2 X, X\rangle,  
\end{equation}
for any $X \in T\Sigma$. Assume now that $\overline{\mathrm{Sect}}\leq0$ and $\overline{\mathrm{Ric}}_f\geq k$. Then
\begin{equation}\label{boundricf}
\ \mathrm{Ric}_{f}\geq k-|\mathbf{A}|^2, 
\end{equation}
and, if $Ind_{f}(\Sigma)<+\infty$ and $k\geq 0$, we obtain that there exists a solution $\varphi>0$ of the differential inequality
\[
\ \Delta_f \varphi+a(x)\varphi\leq 0,
\]
weakly outside a compact set, where $a(x)=|\mathbf{A}|^2-k$. Hence the assumptions in Corollary \ref{finnonfparends} are met and we can conclude that $\Sigma$ has at most finitely many non--$f$--parabolic ends. Applying Theorem \ref{ThWeightSobL1}, we can now get the following
\begin{theorem}\label{mainThTopo}
Let $\Sigma^m$ be a complete $f$--minimal hypersurface isometrically immersed with $Ind_f(\Sigma)<+\infty$ in a 
complete weighted manifold $M_{f}^{m+1}$ with 
$\overline{\mathrm{Sect}}\leq0$ and $\overline{\mathrm{Ric}}_{f}\geq k\geq0$. Suppose furthermore 
that $f\leq f^{*}<+\infty$ and $|\overline{\nabla}f|\in L^{m}(\Sigma_f)$. 
Then $\Sigma$ has finitely many ends.
\end{theorem} 
\begin{proof}
By Theorem \ref{ThWeightSobL1} and using the $f$--minimality, we have that for every $0\leq h\in C_{c}^{\infty}(\Sigma)$
\begin{equation*}
\left[\int_{\Sigma}h^{\frac{m}{m-1}}e^{-f}d\mathrm{vol}_{\Sigma}\right]^{\frac{m-1}{m}}\leq C\left[\int_{\Sigma}|\nabla h|+h|\overline{\nabla}f|e^{-f}d\mathrm{vol}_{\Sigma}\right].
\end{equation*}
Since we are assuming that $|\overline{\nabla}f|\in L^{m}(\Sigma_f)$, for a suitable compact $K$ we can suppose that
\[
\ \left\|\overline{\nabla}f\right\|_{L^{m}(\Sigma\setminus K,\, e^{-f}d\mathrm{vol}_{\Sigma})}< C^{-1}.
\]
Then, applying the H\"older inequality, the term involving $\theta$ can be absorbed in the left--hand side, showing that the $L^1$--Sobolev inequality
\begin{equation*}
\left[\int_{\Sigma}h^{\frac{m}{m-1}}e^{-f}d\mathrm{vol}_{\Sigma}\right]^{\frac{m-1}{m}}\leq D\left[\int_{\Sigma}|\nabla h|e^{-f}d\mathrm{vol}_{\Sigma}\right]
\end{equation*}
holds for every smooth non-negative function compactly supported in $\Sigma\setminus K$ and some constant $D>0$. By Remark \ref{SobL1InfVolEnds} we hence conclude the proof.
\end{proof}

In the discussion just above Theorem \ref{mainThTopo} we needed the hypothesis on $Ind_f(\Sigma)$, jointly with $\overline{\mathrm{Ric}}_f\geq k\geq 0$ in order to guarantee the finiteness of the Morse index of the operator
$-\Delta_f-(|\mathbf{A}|^2-k)$. Note that, on the other hand, in case $k\geq0$ we have even that $\mathrm{Ric}_{f}\geq-|\mathbf{A}|^2$.
To apply Corollary \ref{finnonfparends} it thus suffices to guarantee the finiteness of the Morse index of the operator $L_\mathbf{A}=-\Delta_f-|\mathbf{A}|^2$. In particular, adapting ideas in \cite{LiYauCMP}, we are going to show that this can be done assuming the finiteness of weighted total curvature.

Slightly adapting the proof in \cite{LiYauCMP}, it is easy to obtain the following weighted version of Theorem 2 in \cite{LiYauCMP}.
\begin{lemma}\label{LemFI}
Let $\Sigma^m$, $m\geq 3$, be a complete non--compact Riemannian manifold enjoing the $L^2$--weighted Sobolev inequality
\begin{equation}\label{L2Sobo} \left(\int_{\Sigma}h^{\frac{2m}{m-2}}e^{-f}d\mathrm{vol}_{\Sigma}\right)^{\frac{m-2}{m}}\leq C(m)\left(\int_\Sigma|\nabla h|^2e^{-f}d\mathrm{vol}_{\Sigma}\right)\quad\forall\,h\in C_{c}^{\infty}(\Sigma).
\end{equation}
Let $D\subseteq\Sigma$ be a bounded domain. Suppose $q(x)$ is a positive function defined on $D$ and let $\mu_{k}$ be the $k^{th}$ eigenvalue for
\begin{equation*}
\begin{cases}
\Delta_f\psi(x)=-\mu q(x)\psi(x)&\mathrm{on}\,\,D\\
\left.\psi\right|_{\partial D}\equiv0&
\end{cases}
\end{equation*}
Then
\[
\ \mu_{k}^{\frac{m}{2}}\int_{D}q^{\frac{m}{2}}e^{-f}d\mathrm{vol}_{\Sigma}\geq k \tilde{C}(m).
\]
\end{lemma}

Using the same idea as in \cite{LiYauCMP}, we can prove the following
\begin{proposition}\label{PropFI}
Let $\Sigma^m\to M^{m+1}_{f}$, $m\geq3$, be a complete isometrically immersed hypersurface enjoying the $L^2$--weighted Sobolev inequality \eqref{L2Sobo}. Set $L_\mathbf{A}=-\Delta_f-|\mathbf{A}|^2$. Then
\[
\ Ind^{L_{\mathbf{A}}}(\Sigma)\leq \tilde{C}(m)\int_{\sigma}|\mathbf{A}|^{m}e^{-f}d\mathrm{vol}_{\Sigma}.
\]
\end{proposition}

\begin{proof}
Up to taking an exhaustion of $\Sigma$ by compact domains $\left\{\Omega_i\right\}_{i=1}^{\infty}$, it suffices to show that 
\[
\ Ind^{L_{\mathbf{A}}}(\Omega)\leq \tilde{C}(m)\int_{\Omega}|\mathbf{A}|^{m}e^{-f}d\mathrm{vol}_{\Sigma}
\]
for any given domain $\Omega\subseteq\Sigma$. On the other hand,  consider the eigenvalue problem
\begin{equation}\label{feigen}
\begin{cases}
\Delta_f\psi=-\mu|\mathbf{A}|^2\psi\quad\mathrm{on}\quad\Omega\\
\left.\psi\right|_{\partial\Omega}\equiv0.
\end{cases}
\end{equation}
It is not difficult to prove that
\begin{equation}\label{indLA}
Ind^{L_\mathbf{A}}(\Omega)=\sharp\left\{\mu_{k}\leq 1\,\,|\,\, \mu_{k}\mathrm{\,\,is\,\,an\,\,eigenvalue\,\,of\,\,} \eqref{feigen}\right\}.
\end{equation}
Indeed this follows from the identity
\[
\ \frac{\int\left(|\nabla \psi|^2-|\mathbf{A}|^2\psi^2\right)e^{-f}d\mathrm{vol}_{\Sigma}}{\int \psi^2e^{-f}d\mathrm{vol}_{\Sigma}}=\frac{\int |\mathbf{A}|^2\psi^2e^{-f}d\mathrm{vol}_{\Sigma}}{\int \psi^2 e^{-f}d\mathrm{vol}_{\Sigma}}\left[\frac{\int |\nabla \psi|^2e^{-f}d\mathrm{vol}_{\Sigma}}{\int |\mathbf{A}|^2\psi^2e^{-f}d\mathrm{vol}_{\Sigma}}-1\right],
\]
observing that $$\frac{\int |\nabla \psi|^2e^{-f}d\mathrm{vol}_{\Sigma}}{\int |\mathbf{A}|^2\psi^2e^{-f}d\mathrm{vol}_{\Sigma}}$$ is the quadratic form associated to the operator $-\frac{\Delta_{f}}{|\mathbf{A}|^2}$. 
Hence, if $\mu_k$ is the greatest eigenvalue of \eqref{feigen} less then or equal to $1$, it follows by Lemma \ref{LemFI} that
\[
\ Ind^{L_{\mathbf{A}}}=k\leq \tilde{C}(m)\mu_{k}^{\frac{m}{2}}\int_{\Omega}|\mathbf{A}|^{m}e^{-f}d\mathrm{vol}_{\Sigma}\leq \tilde{C}(m)\int_{\Omega}|\mathbf{A}|^me^{-f}d\mathrm{vol}_{\Sigma}.
\]
\end{proof}
As a consequence of Proposition \ref{PropFI} we can now state the announced corollary of Theorem \ref{mainThTopo}.
\begin{corollary}\label{mainCoroTopo}
Let $\Sigma^m$ be a complete $f$--minimal hypersurface isometrically immersed in a complete weighted manifold $M_{f}^{m+1}$ with $\overline{\mathrm{Sect}}\leq0$ and $\overline{\mathrm{Ric}}_{f}\geq k\geq0$. Assume that $|\mathbf{A}|\in L^{m}(\Sigma_f)$. Suppose furthermore that $f\leq f^{*}<+\infty$ and $|\overline{\nabla}f|\in L^{m}(\Sigma_f)$. Then $\Sigma$ has finitely many ends.
\end{corollary}

\bigskip

\begin{acknowledgement*}
\rm{Part of this work was done while we were visiting the Institut Henri Poincar\'e, Paris. We would like to thank the institute for the warm hospitality.\\
Moreover we are deeply grateful to Stefano Pigola for useful conversations during the preparation of the manuscript. We would also like to thank Marcio Batista and Heudson Mirandola and the anonymous referee for useful comments.}
\end{acknowledgement*}

\bibliographystyle{amsplain}
\bibliography{bib_NotesFMin}
 
\end{document}